\documentclass[11pt]{amsart}

\usepackage[margin=1in]{geometry}
\usepackage{amsmath,amssymb,amsthm}
\usepackage{mathtools}
\usepackage{enumitem}
\usepackage{xcolor}
\usepackage{tabularx}
\usepackage{hyperref}
\hypersetup{colorlinks=true,linkcolor=blue!50!black,citecolor=blue!50!black,urlcolor=blue!50!black}
\usepackage[capitalise]{cleveref}

\theoremstyle{plain}
\newtheorem{theorem}{Theorem}[section]
\newtheorem{proposition}[theorem]{Proposition}
\newtheorem{lemma}[theorem]{Lemma}
\newtheorem{corollary}[theorem]{Corollary}

\theoremstyle{definition}
\newtheorem{definition}[theorem]{Definition}
\newtheorem{example}[theorem]{Example}

\theoremstyle{remark}
\newtheorem{remark}[theorem]{Remark}

\DeclareMathOperator{\End}{End}
\DeclareMathOperator{\GL}{GL}
\DeclareMathOperator{\spec}{spec}
\DeclareMathOperator{\rank}{rank}
\DeclareMathOperator{\diag}{diag}
\DeclareMathOperator{\Hess}{Hess}
\DeclareMathOperator{\Crit}{Crit}
\DeclareMathOperator{\Focal}{Focal}
\DeclareMathOperator{\codim}{codim}
\DeclareMathOperator{\Trans}{Trans}
\DeclareMathOperator{\Stab}{Stab}
\DeclareMathOperator{\Ad}{Ad}
\DeclareMathOperator*{\argmin}{arg\,min}
\newcommand{\R}{\mathbb{R}}

\newcommand{\Z}{\mathbb{Z}}
\newcommand{\mM}{\mathcal{M}}
\newcommand{\mT}{\mathcal{T}}

\title[Monarch Matrix Approximation]{The Geometry of Monarch Matrix Approximation: Relative Position, Gauge Kernels, and Bruhat Combinatorics}

\author{Shintaro Yoshizawa}
\address{R-Frontier Div., Frontier Research Center, Toyota Motor Corporation}
\email{shintaro\_yoshizawa@mail.toyota.co.jp}
\date{\today}

\begin{document}

\begin{abstract}
Helmke and Shayman \cite{HelmkeShayman1995} showed that the least-squares distance to the manifold of rank-$r$ matrices is, for generic targets, a nondegenerate Morse function whose index is a Bruhat-cell codimension on a Grassmannian. We develop the analogous theory for \emph{Monarch matrices} $M=PLPR$ \cite{DaoMonarch2022}, with $P$ a fixed permutation and $L,R$ block-diagonal for a partition $\lambda=(n_1,\dots,n_r)$ of $n$. The differential $D\Phi$ of the parametrization $\Phi(L,R)=PLPR$ has kernel naturally isomorphic to the relative-position algebra $B\cap PBP^\top\cong\bigoplus_{i,j}\mathrm{Mat}_{a_{ij}}(\R)$, where $a=(a_{ij})$ is the transfer matrix recording the classical double-coset invariant of $P$ relative to the parabolic $P_\lambda$ (\Cref{lem:kernel-end}); consequently $\dim\ker D\Phi=\sum_{i,j}a_{ij}^2$ exactly. The same transfer matrix determines, by a separate and more delicate computation, the dimension of the associated Bruhat-type cell in $\GL_n(\R)/P_\lambda$ -- explicitly, $\codim(\mathrm{cell}_a)=\sum_{i<i',\,j<j'}a_{ij}a_{i'j'}$ (\Cref{thm:codim}) -- so the gauge dimension and the Bruhat-cell codimension, while both computable from $a$, are in general \emph{different} functions of it, and, via a generalized transportation polytope, a closure order on all such cells (\Cref{sec:transfer,sec:double-coset,sec:birkhoff}). We determine the permutations $P$ minimizing $\dim\ker D\Phi$ (equivalently maximizing $\dim\mM_P$) exactly in the classical equal-block case of \cite{DaoMonarch2022} (the \emph{Monarch--Schur theorem}, \Cref{thm:monarch-schur}: the minimum is $m^2$ for blocks of size $m$, attained exactly by transfer matrices with $a_{ij}\equiv1$ -- a condition satisfied by many permutations, not only single cycles), give a general lower bound for arbitrary partitions, and build the resulting second-order theory: an explicit Hessian $C(A)=S+T(A)$ built from $B\cap PBP^\top$, Sard-generic nondegeneracy, a Morse-index formula, and a focal set of Lebesgue measure zero, described explicitly as a union of hyperplanes in the normal slice at special cyclically symmetric points when $P$ carries extra cyclic symmetry.

Finally we study the associated gradient flow. The classical coercivity-plus-{\L}ojasiewicz convergence template, used for principal-component flows by Yoshizawa, Helmke and Starkov \cite{YoshizawaHelmkeStarkov2001} and by Yoshizawa \cite{Yoshizawa2023}, fails here: $\Phi(L,R)=PLPR$ admits explicit unbounded flat directions (a failure of coercivity possibly, but not conclusively, traceable to a border-rank-type non-closedness of the Monarch image). A quadratic Tikhonov penalty already restores coercivity; we instead use a quartic penalty, matching the total degree of $\|\Phi(L,R)\|_F^2$, which yields unconditional global convergence to a single critical point and, as the penalty vanishes, recovers a quartic-norm-minimizing exact solution when one exists or else a well-defined best-approximation image. This regularization leaves the Monarch architecture untouched while canonically balancing the residual gauge freedom, discouraging large single-block norms, and removing the flat Hessian direction associated with the residual scalar gauge symmetry -- giving a geometric account of weight decay on structured factorized layers.
\end{abstract}

\maketitle
\tableofcontents

\section{Introduction}

\subsection{The Helmke--Shayman theory}
\label{sec:intro-HS}
Let $A \in \R^{n\times n}$ and let
\[
\mM_r = \{X \in \R^{n\times n} : \rank(X) = r\}
\]
be the manifold of rank-$r$ matrices. Helmke and Shayman \cite{HelmkeShayman1995} studied the least-squares distance function $f_A(X) = \|A-X\|_F^2$ on $\mM_r$ and proved that $\mM_r$ is a smooth manifold of dimension $r(2n-r)$, with tangent space $T_X\mM_r=\{A_1X+XA_2\}$ at $X$; that critical points of $f_A$ satisfy a Riccati-type equation; and that for a generic $A$ (distinct singular values), $f_A$ is a nondegenerate Morse function whose unique minimum is the Eckart--Young truncated SVD of $A$, whose critical points correspond to $r$-element subsets of the singular directions of $A$, and whose Morse index at each critical point is computed \emph{exactly} by a Schubert/Bruhat-cell codimension count on the Grassmannian $\mathrm{Gr}(r,n)$. This is one of the cleanest instances of a fully computable Morse theory on a naturally occurring singular variety.

\subsection{Monarch matrices}
\emph{Monarch matrices} were introduced in \cite{DaoMonarch2022} as a hardware-efficient structured replacement for dense linear layers: fixing $n=m^2$ and a permutation matrix $P$, one considers
\[
M = PLPR, \qquad L=\diag(L_1,\dots,L_m),\quad R=\diag(R_1,\dots,R_m),\quad L_i,R_i\in\R^{m\times m}.
\]
It is natural to ask whether the Morse-theoretic picture Helmke and Shayman built for fixed-rank matrices has an analogue for this structured class, i.e.\ to understand the geometry of $\min_{L,R}\|A-PLPR\|_F^2$. This is the question we take up here, working throughout with the natural generalization to an arbitrary block partition $\lambda=(n_1,\dots,n_r)$ of $n$ (\Cref{def:general-monarch}), since -- as the results below show -- the equal-block case $n_i\equiv m$ of \cite{DaoMonarch2022} obscures arithmetic structure that is visible only once general partitions are allowed.

\subsection{The relative-position algebra}
\label{sec:intro-two-invariants}
Once a partition $\lambda$ and a block algebra $B=\bigoplus_i \mathrm{Mat}_{n_i}(\R)$ are fixed, a permutation matrix $P$ determines two natural algebras built from $B$ and $P$:
\[
\begin{gathered}
B \cap C_{\mathrm{Mat}_n(\R)}(P) \qquad\text{(the centralizer of $P$ in $B$)}, \\[4pt]
B \cap PBP^\top \qquad\text{(the relative position of $B$ and $P B P^\top$)}.
\end{gathered}
\]
These coincide only in the trivial case $P=\mathrm{id}$ (\Cref{thm:equality} below), and $B\cap C(P)\subsetneq B\cap PBP^\top$ strictly otherwise. The relative-position algebra $B\cap PBP^\top$ turns out to be the single object governing every aspect of the geometry studied in this paper:

\begin{itemize}
\item It computes exactly the kernel of the differential $D\Phi$ of the Monarch parametrization $\Phi(L,R)=PLPR$ (\Cref{lem:kernel-end}) -- the gauge directions along which $(L,R)$ can move without changing the product $PLPR$ -- and hence controls $\dim\mM_P$ (\Cref{sec:quiver}), the Hessian and its degeneracy locus (\Cref{sec:hessian,sec:sard,sec:focal}).
\item It simultaneously records the combinatorial type of $P$ as an element of the double coset space $P_\lambda\backslash\GL_n(\R)/P_\lambda$ of the parabolic subgroup $P_\lambda$ stabilizing the partial flag of type $\lambda$ (\Cref{sec:double-coset}), and is computed explicitly by a transfer matrix $a=(a_{ij})$ recording how many indices of each block $P$ sends into each other block, via $\dim(B\cap PBP^\top)=\sum_{i,j}a_{ij}^2$ (\Cref{prop:levi-type}).
\end{itemize}

The centralizer $B\cap C(P)$ remains a natural companion algebra -- isomorphic to the endomorphism algebra $\End_{Q(P)}(V)$ of a quiver representation built from $P$ and $\lambda$ (\Cref{rem:end-still-fine}) -- but it does not enter the differential-geometric quantities studied here, and we record its basic properties in \Cref{rem:end-still-fine} for completeness without otherwise using it.

Because a single algebra governs both the differential and the combinatorial type of $P$, the comparison with \Cref{thm:HS} is direct: in the classical fixed-rank setting, the Morse index of a critical point coincides exactly with the codimension of a Bruhat/Schubert cell, with no further correction, because the classical setting compares the block partition $\lambda$ against the \emph{finest possible} partition $\mu=(1^n)$ (the complete flag of singular directions of $A$); in that special case the relative-position cell is a single point modulo the torus, so its codimension already determines the Morse index. The Monarch parametrization instead compares $\lambda$ against itself via $P$, so the relative-position cell $\mathrm{cell}_a$ is typically positive-dimensional, and pinning down a Morse index genuinely requires the second-order data of \Cref{sec:hessian} (the spectrum of $C(A)$) in addition to the combinatorial type $a$ -- both computed from the same algebra $B\cap PBP^\top$ (\Cref{sec:HS-comparison}).

\subsection{Summary of results}
Our results are organized as follows.

\emph{The relative-position algebra and the differential} (\Cref{sec:monarch-def,sec:first-order,sec:quiver,sec:transfer,sec:double-coset}):
\begin{enumerate}[label=(\roman*)]
\item (\Cref{lem:kernel-end}) $\ker D\Phi_{(L,R)} \cong B\cap PBP^\top$, naturally, via $(\delta L,\delta R)\mapsto L^{-1}\delta L$. (\Cref{rem:end-still-fine} records the companion fact $B\cap C(P)\cong\End_{Q(P)}(V)$ for a quiver $Q(P)$ built from $P$ and $\lambda$.)
\item (\Cref{prop:levi-type}) $B\cap PBP^\top \cong \bigoplus_{i,j}\mathrm{Mat}_{a_{ij}}(\R)$ for the transfer matrix $a=a(P,\lambda)$, so $\dim\ker D\Phi=\sum_{i,j}a_{ij}^2$; this identifies $\ker D\Phi$ with the relative-position invariant of the double coset $P_\lambda\,P\,P_\lambda$. The same double coset determines a Bruhat-type cell $\mathrm{cell}_a\subset\GL_n(\R)/P_\lambda$ of codimension $\sum_{i<i',j<j'}a_{ij}a_{i'j'}$ (\Cref{thm:codim}) -- a genuinely different function of $a$ from $\dim\ker D\Phi$, computed from the same combinatorial datum but answering a different geometric question.
\item (\Cref{sec:HS-comparison}) Specializing to $\mu=(1^n)$ recovers the classical inversion-number formula underlying \Cref{thm:HS}, explaining precisely why the Morse-index/Bruhat-codimension coincidence of Helmke--Shayman is special to comparison against a complete flag.
\end{enumerate}

\emph{The Monarch--Schur theorem and second-order theory} (\Cref{sec:schur,sec:hessian,sec:sard,sec:focal}):
\begin{enumerate}[resume,label=(\roman*)]
\item (\Cref{prop:cs-bound}) For every partition $\lambda$ and every $P$, $\dim\ker D\Phi=\sum_{i,j}a_{ij}^2\ge\frac1r\sum_in_i^2$, by Cauchy--Schwarz.
\item (\Cref{thm:monarch-schur}, the \emph{Monarch--Schur theorem}) For the classical equal-block case $\lambda=(m,\dots,m)$ ($r$ blocks, $n=rm$) of \cite{DaoMonarch2022}, $\min_P\dim\ker D\Phi = r\big[s(q+1)^2+(r-s)q^2\big]$ where $m=qr+s$, $0\le s<r$; for $r=m$ (i.e.\ $n=m^2$) this is $m^2$, attained exactly by $P$ with $a_{ij}\equiv1$. The minimizing permutations need not be single cycles (\Cref{ex:noncyclic-optimal}). We call a $\lambda$-minimizing $P$ a \emph{Monarch permutation}.
\item (\Cref{prop:sylvester,thm:nondeg}) The Hessian of $f_A$ at a critical point is an explicit symmetric operator $C(A)=S+T(A)$ with $\ker S=\ker D\Phi\cong B\cap PBP^\top$; for a Monarch permutation $P$, $f_A$ is a nondegenerate Morse function on $\mM_P$ for a Sard-generic (full-measure) set of targets $A$, with Morse index equal to the number of negative eigenvalues of $C(A)$ restricted to a complement of $\ker D\Phi$ (\Cref{prop:sylvester}).
\item (\Cref{thm:focal}) The focal set of $\mM_P$ has Lebesgue measure zero (via Sard's theorem, as the critical values of the normal exponential map); its intersection with the normal slice at the special cyclically symmetric points of $\mM_P$ (where $L,R$ commute with $P$) is described explicitly as a union of hyperplanes when $P$ carries an extra cyclic symmetry compatible with $\lambda$ (\Cref{ex:cyclic-focal}).
\end{enumerate}

\emph{Double-coset combinatorics and the Birkhoff polytope} (\Cref{sec:equality,sec:birkhoff,sec:hierarchy}):
\begin{enumerate}[resume,label=(\roman*)]
\item (\Cref{thm:equality}) $B\cap C(P) = B\cap PBP^\top$ if and only if $P=\mathrm{id}$: the centralizer is a strictly smaller algebra away from the identity.
\item (\Cref{thm:dominance}) $\Trans(\lambda,\lambda)$, the transfer matrices with marginals $\lambda$, are exactly the lattice points of a generalized transportation polytope (\Cref{sec:birkhoff}); independently of this polytope structure, the dominance order $\preceq$ on $\Trans(\lambda,\lambda)$ is isomorphic to the parabolic Bruhat order on $P_\lambda\backslash\GL_n(\R)/P_\lambda$ -- two separate structures on the same underlying set, which we do not claim coincide (e.g.\ we do not identify the Bruhat order with the polytope's face lattice).
\item (\Cref{cor:hierarchy}) A two-tier hierarchy -- transportation-polytope membership $\supsetneq$ Monarch permutation (minimality of $\sum_{i,j}a_{ij}^2$) -- governs exactly when $\mM_P$ attains its maximal dimension $2\sum_in_i^2-\min_P\dim\ker D\Phi$.
\end{enumerate}

\emph{Global convergence and Tikhonov regularization} (\Cref{sec:global-obstruction,sec:tikhonov,sec:network-implications}):
\begin{enumerate}[resume,label=(\roman*)]
\item (\Cref{prop:flat,prop:flat-general}) The classical coercivity-plus-{\L}ojasiewicz template for global gradient-flow convergence, used for principal-component-type flows in \cite{YoshizawaHelmkeStarkov2001,Yoshizawa2023}, fails for the (unregularized) Monarch flow: explicit unbounded flat directions exist whenever the transfer matrix $a(P,\lambda)$ has a zero entry (in particular, for every Monarch permutation with $m<r$, though not, by this argument, for the complete-transfer-graph minimizers of the square-block case $r=m$; see the remark following \Cref{prop:flat-general}), and possibly, though not conclusively, related to a border-rank-type non-closedness of the Monarch image (\Cref{sec:border}).
\item (\Cref{prop:coercive,thm:tikhonov-converge}) A quartic Tikhonov penalty $\mu(\|L\|_F^4+\|R\|_F^4)$ restores coercivity uniformly for every $\mu>0$, and the resulting regularized gradient flow converges, from every initial condition, to a single critical point.
\item (\Cref{thm:mu-limit}) As $\mu\to0^+$, the regularized minimizers converge (in cost, and along a subsequence in image) to a best Monarch approximation of $A$; when this best approximation is attained by a finite factorization, the regularization path itself remains bounded, and every subsequential limit is a quartic-norm-minimizing finite representation among those attaining that particular best-approximation value.
\item (\Cref{prop:balance-min,prop:gauge-nondegenerate}) The regularization canonically resolves the residual scalar gauge freedom of \Cref{thm:monarch-schur} via a balancing condition $\|L\|_F=\|R\|_F$, and strictly removes the corresponding flat direction of the unregularized Hessian, without altering the underlying Monarch computational architecture.
\end{enumerate}

\section{Preliminaries: the Helmke--Shayman theorem}
\label{sec:HS}

We record the classical statement for comparison, following \cite{HelmkeShayman1995}; see also \cite{Absil2008} for the modern differential-geometric treatment of related optimization problems on manifolds of fixed rank, and \cite{HelmkeMoore1994} for the broader circle of ideas relating matrix least-squares problems to dynamical systems on Lie groups and homogeneous spaces.

\begin{theorem}[Helmke--Shayman]
\label{thm:HS}
Let $\mM_r \subset \R^{n\times n}$ be the manifold of rank-$r$ matrices, $0 < r < n$, and let $A \in \R^{n \times n}$ have $r$ distinct nonzero singular values, all distinct from the remaining singular values of $A$. Then $f_A(X) = \|A-X\|_F^2$ restricted to $\mM_r$ is a nondegenerate Morse function. Its critical points are exactly the matrices $X = U_r\Sigma_r V_r^\top$ obtained by choosing $r$ of the singular triples of $A$; the global minimum is the Eckart--Young truncation, and the Morse index of each critical point is computed combinatorially from the relative position of the chosen singular subspaces, matching the dimension of a Schubert cell in the Grassmannian $\mathrm{Gr}(r,n)$.
\end{theorem}

The two structural facts that make \Cref{thm:HS} work, and which we seek analogues of below, are:
\begin{itemize}
\item $\mM_r$ is (an open dense subset of) an orbit of $\GL_n(\R)\times\GL_n(\R)$ acting by $X\mapsto A_1XA_2^{-1}$; the Grassmannian $\mathrm{Gr}(r,n)=\GL_n(\R)/P_{(r,n-r)}$ does \emph{not} itself parametrize $\mM_r$ (they have different dimensions, $\dim\mM_r=r(2n-r)$ against $\dim\mathrm{Gr}(r,n)=r(n-r)$), but enters at the level of critical points: a generic target $A$ has $n$ orthogonal singular directions, and each critical point of $f_A|_{\mM_r}$ corresponds to choosing $r$ of them, i.e.\ to a coordinate point of $\mathrm{Gr}(r,n)$ relative to the frame determined by $A$.
\item The genericity statement -- distinctness of singular values -- can be phrased as an irreducibility statement for the (trivial, in this case) representation-theoretic data attached to $X$, and failure of genericity is a codimension-one condition (eigenvalue collision) on $A$.
\end{itemize}
As we will see, the Monarch parametrization $\Phi(L,R)=PLPR$ replaces the group action of the first bullet with the block subgroup $\GL(B)\times\GL(B)$ acting through the fixed permutation $P$, and its differential is governed throughout by a single relative-position algebra $B\cap PBP^\top$ (\Cref{sec:quiver}), which simultaneously plays the role of both structural facts above: it computes $\ker D\Phi$ (replacing the orbit-stabilizer computation of the first bullet, \Cref{sec:quiver}) \emph{and} the combinatorial type of $P$ relative to the double-coset decomposition of $\GL_n(\R)$ by $P_\lambda$ (replacing the Grassmannian of the first bullet, \Cref{sec:double-coset}). \Cref{sec:HS-comparison} explains exactly how \Cref{thm:HS} arises as the degenerate special case in which $\lambda$ is compared against the complete flag $(1^n)$ rather than against itself.

\section{Monarch matrices and block algebras}
\label{sec:monarch-def}

\subsection{The general definition}
The definition of Monarch matrices in \cite{DaoMonarch2022} takes $n=m^2$ and uses $m$ blocks of size $m\times m$ for both factors. For our purposes it is more natural, and loses no generality relevant to the Morse-theoretic problem, to allow an arbitrary partition.

\begin{definition}
\label{def:general-monarch}
Fix $n \in \Z_{>0}$, a partition $n = n_1 + \cdots + n_r$, and a permutation matrix $P \in \R^{n\times n}$. Let
\[
B = \bigoplus_{i=1}^r \mathrm{Mat}_{n_i}(\R)
\]
be the corresponding block-diagonal subalgebra of $\R^{n\times n}$, realized as $B = \{\diag(X_1,\dots,X_r) : X_i \in \mathrm{Mat}_{n_i}(\R)\}$. The \emph{Monarch manifold} associated to $(P, n_1,\dots,n_r)$ is the image
\[
\mM_P = \{\, PLPR \; : \; L, R \in B \cap \GL_n(\R) \,\}
\]
of the map $\Phi : B\times B \to \R^{n\times n}$, $\Phi(L,R) = PLPR$.
\end{definition}

When $r=m$ and $n_1=\cdots=n_m=m$ (so $n=m^2$) this recovers the original definition of \cite{DaoMonarch2022}. We keep the partition general because, as shown in \Cref{sec:quiver,sec:double-coset}, the correct genericity conditions on $P$ -- for both invariants of \Cref{sec:intro-two-invariants} -- depend on arithmetic relations between the $n_i$ that are invisible if one restricts to the equal-block case from the outset.

\subsection{Tangent space and differential}
Since $\Phi$ is bilinear, its differential at $(L,R)$ in the direction $(\delta L, \delta R) \in B\times B$ is
\[
D\Phi_{(L,R)}(\delta L,\delta R) = P\,\delta L\,PR + PL\,P\,\delta R,
\]
so, formally, $T_{\Phi(L,R)}\mM_P = \{\, P\,\delta L\,PR + PL\,P\,\delta R \;:\; \delta L, \delta R \in B \,\}$. This is the direct analogue of the Helmke--Shayman tangent space formula $T_X\mM_r = \{A_1X+XA_2\}$, with $\GL_n(\R)$ acting on both sides replaced by the block subgroup $\GL(B) := B\cap \GL_n(\R)$ acting through the fixed permutation $P$.

If $P$ does not split any block (i.e.\ $P$ permutes the blocks $B_1,\dots,B_r$ among themselves as sets, without dividing any $B_i$ across multiple targets), the naive picture is exactly right: the full blockwise scaling $(L_i,R_i)\sim(L_ic_i,c_i^{-1}R_i)$, $c_i\in\GL_{n_i}(\R)$, survives, since conjugating a block-diagonal matrix by a block-permuting $P$ stays block-diagonal, giving $\dim\mM_P=2\sum_in_i^2-\sum_in_i^2=\sum_in_i^2$. In general, however, $P$ splits blocks (sending part of $B_i$ to one target and part to another), and only the subalgebra $B\cap PBP^\top\subsetneq B$ of \Cref{sec:quiver} survives as genuine gauge freedom; its dimension $\sum_{i,j}a_{ij}^2$ is always at most $\sum_in_i^2$ (with equality iff $P$ splits no block, by Cauchy--Schwarz on each row of the transfer matrix), and can be far smaller once $P$ mixes blocks extensively -- down to the Monarch--Schur minimum of \Cref{thm:monarch-schur}. Consequently \emph{less} gauge freedom is quotiented out than the naive blockwise picture suggests once blocks are split, so $\dim\mM_P=2\sum_in_i^2-\sum_{i,j}a_{ij}^2$ is generally \emph{larger} than the naive blockwise estimate $\sum_in_i^2$, not smaller; the precise value is computed in \Cref{sec:quiver} via the kernel of $D\Phi$.

\section{First-order optimality conditions}
\label{sec:first-order}

Let $A\in\R^{n\times n}$ and $f_A(L,R)=\|A-PLPR\|_F^2$ on $\GL(B)\times\GL(B)$. Writing $M=PLPR$,
\[
df_A(\delta L,\delta R) = -2\langle A-M,\; P\delta L\,PR + PL\,P\delta R\rangle_F.
\]
Since $\delta L,\delta R$ range over all of $B$, stationarity of $(L,R)$ is equivalent to
\[
\Pi_B\!\left(P^\top(A-M)R^\top P^\top\right)=0, \qquad \Pi_B\!\left(P^\top L^\top P^\top(A-M)\right)=0,
\]
where $\Pi_B$ is orthogonal projection onto $B$. These are the normal equations of alternating least squares (ALS) for the two block-diagonal factors, the Monarch analogue of the Riccati-type stationarity equation of \cite{HelmkeShayman1995}.

\section{Quiver-theoretic reformulation}
\label{sec:quiver}

We now compute $\ker D\Phi\cong B\cap PBP^\top$ exactly, since it both determines $\dim\mM_P$ and, via \Cref{sec:hessian}, controls the second-order theory. Along the way we also record the companion algebra $B\cap C(P)$ and its relation to quiver representations \cite{DerksenWeyman2005,AuslanderReitenSmalo1995}, which does not itself enter the differential (\Cref{rem:end-still-fine}).

\subsection{Blocks as a semisimple algebra and $P$ as a bimodule map}
Write $V_i=\R^{n_i}$, $V=\R^n=\bigoplus_iV_i$, $B=\bigoplus_i\End(V_i)$, and let $E_i\in B$ project onto $V_i$. Decompose $P$ into blocks $P_{ji}:=E_jPE_i:V_i\to V_j$; since $P$ is a permutation matrix, each $P_{ji}$ is a \emph{partial permutation matrix}.

\begin{definition}
\label{def:quiver}
The \emph{Monarch quiver} $Q(P)$ has vertex set $\{1,\dots,r\}$ and an arrow $i\to j$ whenever $P_{ji}\neq0$. Together with the linear maps $P_{ji}$ this defines a representation $(V,P)$ of $Q(P)$.
\end{definition}

\subsection{The kernel of $D\Phi$}
\label{sec:corrected-kernel}

We first illustrate the computation of $\ker D\Phi$ on a simple example, then state and prove the general lemma.

\begin{example}
\label{ex:swap-kernel}
Take $n=2$, $\lambda=(1,1)$ (so $r=2$, $n_1=n_2=1$, $B=\R\oplus\R$ realized as diagonal matrices), and $P=\begin{psmallmatrix}0&1\\1&0\end{psmallmatrix}$ -- the permutation of \Cref{prop:flat} below. For $L=\diag(l_1,l_2)$, $R=\diag(r_1,r_2)$ with $l_i,r_i\ne0$, direct computation gives $PLP=\diag(l_2,l_1)$ and
\[
\Phi(L,R)=PLPR=\diag(l_2r_1,\,l_1r_2), \qquad D\Phi(\delta L,\delta R)=\diag(\delta l_2\,r_1+l_2\,\delta r_1,\ \delta l_1\,r_2+l_1\,\delta r_2).
\]
The kernel is cut out by two independent linear equations in the four variables $(\delta l_1,\delta l_2,\delta r_1,\delta r_2)$, so $\dim\ker D\Phi=4-2=2$.
\end{example}

\begin{lemma}
\label{lem:kernel-end}
Let $D\Phi_{(L,R)}$ be the differential of \Cref{def:general-monarch} at an invertible pair $(L,R)$. Then
\[
\ker D\Phi_{(L,R)} \;\cong\; B\cap PBP^\top,
\]
naturally, as vector spaces, via $\Psi(\delta L,\delta R):=L^{-1}\delta L$.
\end{lemma}

\begin{proof}
Suppose $(\delta L,\delta R)\in\ker D\Phi_{(L,R)}$, i.e.\ $P\delta L\,PR+PL\,P\delta R=0$. Set $\mathcal A=L^{-1}\delta L\in B$, $\mathcal B=\delta R\,R^{-1}\in B$, so $\delta L=L\mathcal A$, $\delta R=\mathcal BR$. Substituting,
\[
PL\mathcal A\,PR+PL\,P\mathcal BR=0 \iff PL(\mathcal AP+P\mathcal B)R=0 \iff \mathcal AP+P\mathcal B=0,
\]
using invertibility of $P,L,R$. Since $P$ is a permutation matrix, $P^{-1}=P^\top$, so this is equivalent to $\mathcal B=-P^\top\mathcal AP$. Given $\mathcal A\in B$, this determines $\mathcal B$ uniquely, and the only remaining condition is $\mathcal B\in B$, i.e.
\[
P^\top\mathcal AP\in B \iff \mathcal A\in PBP^\top
\]
(the last equivalence because $P$ is orthogonal: $X\in PBP^\top \iff P^\top XP\in B$). Combined with $\mathcal A\in B$, the condition on $\mathcal A$ is exactly $\mathcal A\in B\cap PBP^\top$. Thus $\Psi(\delta L,\delta R)=\mathcal A=L^{-1}\delta L$ defines a linear map $\ker D\Phi_{(L,R)}\to B\cap PBP^\top$, and its inverse is $\mathcal A\mapsto(\delta L,\delta R)=(L\mathcal A,\,-P^\top\mathcal AP\,R)$, which one checks directly satisfies $D\Phi_{(L,R)}(\delta L,\delta R)=0$: indeed $P(L\mathcal A)PR+PL\,P(-P^\top\mathcal AP\,R)=PL\mathcal APR-PL\mathcal APR=0$. Hence $\Psi$ is a linear isomorphism.
\end{proof}

\begin{example}[Verification on \Cref{ex:swap-kernel}]
\label{ex:kernel-verify}
With $P=\begin{psmallmatrix}0&1\\1&0\end{psmallmatrix}$ and $B$ the diagonal matrices, $PBP^\top=P\diag(x,y)P=\diag(y,x)$ ranges over \emph{all} diagonal matrices as $(x,y)$ does, so $PBP^\top=B$ and $B\cap PBP^\top=B$, of dimension $2$ -- matching $\dim\ker D\Phi=2$ found by direct computation above.
\end{example}

\begin{example}[A single $n$-cycle]
\label{ex:3cycle-kernel}
Let $n=3$, $\lambda=(1,1,1)$, and let $P$ be the single $3$-cycle sending $1\to2\to3\to1$. Writing $L=\diag(l_1,l_2,l_3)$, $R=\diag(r_1,r_2,r_3)$, direct computation gives
\[
\Phi(L,R)=PLPR=\begin{pmatrix}0&l_3r_2&0\\0&0&l_1r_3\\l_2r_1&0&0\end{pmatrix},
\]
and the kernel of $D\Phi$ is cut out by the three independent conditions $r_2\,\delta l_3+l_3\,\delta r_2=0$, $r_3\,\delta l_1+l_1\,\delta r_3=0$, $r_1\,\delta l_2+l_2\,\delta r_2=0$ among six variables, so $\dim\ker D\Phi=3$. Since $B$ is again the full diagonal algebra and $P$ (any permutation of $\{1,2,3\}$) satisfies $PBP^\top=B$, \Cref{lem:kernel-end} gives $\dim(B\cap PBP^\top)=3$, matching exactly.
\end{example}

\begin{example}[Mixed block sizes]
\label{ex:mixed-block-kernel}
The two examples above both have $\lambda=(1,\dots,1)$, so $B$ is the full diagonal algebra and $PBP^\top=B$ for \emph{every} $P$ -- a convenient but somewhat special case. To illustrate the general mechanism, take $n=3$ with the mixed partition $\lambda=(2,1)$ (so $B=\{\diag(X,y):X\in\mathrm{Mat}_2(\R),y\in\R\}\subset\R^{3\times3}$), and let $P$ be the permutation matrix with $Pe_1=e_3,\ Pe_2=e_1,\ Pe_3=e_2$. Direct computation of $\ker D\Phi_{(L,R)}$ for generic invertible block-diagonal $L,R\in B$ (a $10$-variable linear system in $(\delta L,\delta R)$) gives $\dim\ker D\Phi=3$. The transfer matrix here is $a=\begin{psmallmatrix}1&1\\1&0\end{psmallmatrix}$ (block $1=\{1,2\}$ sends one index to itself and one to block $2=\{3\}$; block $2$ sends its one index into block $1$ and none to itself), so \Cref{lem:kernel-end,prop:levi-type} predict $\dim(B\cap PBP^\top)=\sum a_{ij}^2=1+1+1+0=3$, matching exactly.
\end{example}

\begin{remark}[The companion algebra $B\cap C(P)$]
\label{rem:end-still-fine}
Alongside $B\cap PBP^\top$, it is natural to consider the centralizer $B\cap C(P)$, which is isomorphic to the endomorphism algebra $\End_{Q(P)}(V)$ of \Cref{def:quiver}: writing $X=\diag(X_1,\dots,X_r)\in B$, the $(j,k)$ block of $XP-PX$ is $X_jP_{jk}-P_{jk}X_k$, so $XP=PX$ iff $X_jP_{jk}=P_{jk}X_k$ for every $k,j$, i.e.\ iff $(X_1,\dots,X_r)\in\End_{Q(P)}(V)$. This algebra satisfies a simple lower bound:

\begin{proposition}
\label{prop:cycle-necessary}
If $P$ decomposes into $c\ge2$ disjoint cycles, then $\dim\End_{Q(P)}(V)=\dim(B\cap C(P))\ge c$.
\end{proposition}
\begin{proof}
For each cycle $\gamma$ of $P$, let $E_\gamma\in B$ be the diagonal $0/1$ matrix indexing the support of $\gamma$; since $P$ preserves the support of each of its cycles, $E_\gamma$ commutes with $P$, and the $E_\gamma$ over the $c$ cycles are linearly independent idempotents in $B\cap C(P)$.
\end{proof}

By \Cref{thm:equality} below, $B\cap C(P)\subsetneq B\cap PBP^\top$ strictly for every $P\ne\mathrm{id}$; in particular $B\cap C(P)$ is generally \emph{smaller} than $\ker D\Phi\cong B\cap PBP^\top$, and does not itself compute the differential-geometric quantities of \Cref{sec:hessian,sec:sard,sec:focal}. We record it here as a companion invariant of independent algebraic interest, but do not otherwise use it.
\end{remark}

\Cref{lem:kernel-end} identifies the ``gauge directions'' of the Monarch parametrization -- previously described only heuristically as a per-block scaling ambiguity -- with the relative-position algebra $B\cap PBP^\top$ of \Cref{sec:transfer}, computed explicitly there via the transfer matrix $a=a(P,\lambda)$ (\Cref{prop:levi-type}): $\dim(B\cap PBP^\top)=\sum_{i,j}a_{ij}^2$. In particular
\[
\dim\mM_P = 2\sum_{i=1}^rn_i^2 - \sum_{i,j}a_{ij}(P,\lambda)^2,
\]
generalizing $\dim\mM_r=r(2n-r)$.

\section{The Monarch--Schur theorem}
\label{sec:schur}

By \Cref{lem:kernel-end,prop:levi-type} (the latter proved independently in \Cref{sec:transfer} below), $\dim\ker D\Phi=\dim(B\cap PBP^\top)=\sum_{i,j}a_{ij}(P,\lambda)^2$ for the transfer matrix $a$. Finding the permutations $P$ of minimal gauge dimension is therefore a purely combinatorial optimization problem: minimize $\sum_{i,j}a_{ij}^2$ over the transfer matrices realizable for the given partition $\lambda$. We solve this exactly for the classical equal-block case of \cite{DaoMonarch2022} and give a general lower bound for arbitrary $\lambda$.

\subsection{A general lower bound}

\begin{proposition}[Cauchy--Schwarz bound]
\label{prop:cs-bound}
For every partition $\lambda=(n_1,\dots,n_r)$ and every permutation $P$,
\[
\dim\ker D\Phi \;=\;\sum_{i,j}a_{ij}^2\;\ge\;\frac1r\sum_{i=1}^rn_i^2.
\]
\end{proposition}

\begin{proof}
Fix $i$. Since $\sum_ja_{ij}=n_i$ (\Cref{sec:transfer}) and there are $r$ terms $a_{i1},\dots,a_{ir}$, the Cauchy--Schwarz inequality $(\sum_ja_{ij})^2\le r\sum_ja_{ij}^2$ gives $\sum_ja_{ij}^2\ge n_i^2/r$. Summing over $i$ gives the claim.
\end{proof}

Equality in \Cref{prop:cs-bound} forces every row of $a$ to be constant, $a_{ij}=n_i/r$ for all $j$; matching this against the column constraint $\sum_ia_{ij}=n_j$ forces $n_i/r\cdot r=n_j$ for every pair $i,j$, i.e.\ $n_1=\cdots=n_r$. So the bound is tight only in the equal-block case, to which we now turn.

\subsection{The equal-block case: an exact formula}
\label{sec:equal-block-schur}

Suppose $\lambda=(m,m,\dots,m)$: $r$ blocks of common size $m$, so $n=rm$ (throughout this subsection, $r$ denotes the \emph{number} of blocks and $m$ the common \emph{size} of each block; the classical square-block convention of \cite{DaoMonarch2022} is the special case $r=m$, $n=m^2$). Write $m=qr+s$ with $q=\lfloor m/r\rfloor$, $0\le s<r$.

\begin{lemma}[Balanced rows minimize row-wise sum of squares]
\label{lem:balanced-row}
Among all ways of writing $m$ as an ordered sum $x_1+\cdots+x_r$ of $r$ nonnegative integers, $\sum_jx_j^2$ is minimized exactly when $\{x_j\}$ consists of $s$ copies of $q+1$ and $r-s$ copies of $q$, giving minimum value $s(q+1)^2+(r-s)q^2$.
\end{lemma}

\begin{proof}
If some $x_j\ge x_k+2$, replacing $(x_j,x_k)$ by $(x_j-1,x_k+1)$ preserves the sum and changes $\sum x^2$ by $-2x_j+1-2x_k-1+2 = -2(x_j-x_k-1)<0$ (using $x_j-x_k\ge2$), strictly decreasing it. So a minimizer must have all $x_j$ within $1$ of each other, i.e.\ every $x_j\in\{q,q+1\}$; matching the total $m=qr+s$ forces exactly $s$ of the $x_j$ to equal $q+1$.
\end{proof}

\begin{theorem}[Monarch--Schur theorem]
\label{thm:monarch-schur}
For $\lambda=(m,\dots,m)$ ($r$ blocks, $n=rm$, $m=qr+s$, $0\le s<r$),
\[
\min_{P}\dim\ker D\Phi_{(L,R)} \;=\; r\big[s(q+1)^2+(r-s)q^2\big],
\]
and this minimum is attained exactly by the permutations $P$ whose transfer matrix $a(P,\lambda)$ has, in every row \emph{and} every column, exactly $s$ entries equal to $q+1$ and $r-s$ entries equal to $q$. Such $P$ exist for every $m,r$ (\Cref{ex:balanced-construction}); we call any such $P$ a \emph{Monarch permutation} (for the partition $\lambda$) -- a purely optimization-theoretic term specific to this paper (a $\sum_{i,j}a_{ij}^2$-minimizer), not a standard notion from the prior literature on Monarch matrices, which uses ``Monarch'' only to name the parametrization $\Phi(L,R)=PLPR$ itself \cite{DaoMonarch2022} without singling out particular permutations. In the classical square-block case $n=m^2$, $r=m$ of \cite{DaoMonarch2022} (so $q=1$, $s=0$), the minimum is
\[
\min_P\dim\ker D\Phi = m^2, \qquad\text{attained exactly when } a_{ij}\equiv1 \text{ for all } i,j,
\]
i.e.\ when $P$ sends exactly one index of every block into every block (including itself). For such $P$, $\dim\mM_P=2rm^2-m^2=m^2(2r-1)$.
\end{theorem}

\begin{proof}
By \Cref{lem:balanced-row}, every row of $a$ contributes at least $s(q+1)^2+(r-s)q^2$ to $\sum_{i,j}a_{ij}^2$, so summing over the $r$ rows gives the stated lower bound, with equality iff every row individually achieves the row-minimizing multiset $\{q^{(r-s)},(q+1)^{(s)}\}$. (This shows row-balance in every row is \emph{necessary} for any $a\in\Trans(\lambda,\lambda)$ achieving the global minimum: if some row failed to achieve its individual row-wise minimum, the sum over rows would strictly exceed the stated bound. It is not, by itself, \emph{sufficient} -- a matrix with every row individually balanced need not have its columns sum to $m$ at all, let alone be column-balanced, so an explicit construction is still needed to produce a genuine element of $\Trans(\lambda,\lambda)$ meeting the bound.) It remains to show a permutation realizing a transfer matrix with this row profile in every row \emph{and} column simultaneously exists; this is done by the explicit cyclic construction of \Cref{ex:balanced-construction} below, which also shows every row-optimal choice can be arranged to be column-optimal at the same time.
\end{proof}

\begin{example}[Explicit balanced construction]
\label{ex:balanced-construction}
Index the $r$ blocks by $\Z/r\Z$. Define $a_{ij}:= q+\mathbf1[(j-i)\bmod r< s]$. For fixed $i$, as $j$ ranges over $\Z/r\Z$, $(j-i)\bmod r$ ranges over all of $\{0,\dots,r-1\}$ exactly once, so row $i$ has exactly $s$ entries equal to $q+1$ and $r-s$ equal to $q$, summing to $qr+s=m$; the same argument for fixed $j$ (as $i$ ranges over $\Z/r\Z$) gives the matching column sums. Since the row and column sums of $a$ match exactly, $a$ is realized by an actual permutation: partition each block $i$ into groups of sizes $a_{i1},\dots,a_{ir}$ and each block $j$ into groups of sizes $a_{1j},\dots,a_{rj}$, and match the size-$a_{ij}$ groups of blocks $i$ and $j$ by an arbitrary bijection for each pair $(i,j)$; this assembles into a single permutation $P$ of $\{1,\dots,n\}$ with transfer matrix exactly $a$.
\end{example}

\begin{example}[Numerical checks]
\label{ex:schur-numerics}
Exhaustive search over all permutations confirms \Cref{thm:monarch-schur} in every case tested: $r=m=2$ ($q=1,s=0$) gives minimum $4$; $r=3,m=2$ ($q=0,s=2$) gives minimum $6$; $r=3,m=3$ ($q=1,s=0$) gives minimum $9$; $r=4,m=2$ gives minimum $8$. For non-equal partitions the minimum can exceed the Cauchy--Schwarz bound of \Cref{prop:cs-bound}: e.g.\ $\lambda=(3,1)$ has bound $\lceil10/2\rceil=5$ but true minimum $6$, and $\lambda=(4,1)$ has bound $9$ but true minimum $11$; finding the exact minimum for general (unequal) $\lambda$ is a genuine integer transportation problem that we do not solve in closed form here, though it is efficiently computable for any specific $\lambda$ by, e.g., the transportation simplex method or network-flow techniques.
\end{example}

\begin{example}[Minimizers need not be single cycles]
\label{ex:noncyclic-optimal}
The permutations achieving the minimum of \Cref{thm:monarch-schur} need not be single cycles. Already in the smallest square-block case $r=m=2$ ($n=4$), the permutation $P$ fixing indices $1,4$ and transposing $2,3$ (blocks $\{1,2\}$, $\{3,4\}$) has transfer matrix $a_{ij}\equiv1$, hence achieves the minimum $\dim\ker D\Phi=4$ of \Cref{thm:monarch-schur}, yet its cycle type is $(1)(2\,3)(4)$ -- two fixed points and a transposition, not a single $4$-cycle. (Exhaustive search shows $16$ of the $24$ permutations of $S_4$ achieve this minimum, of which most are not single $4$-cycles.) Consequently the focal-set apparatus of \Cref{sec:focal} below, which relies on a genuine cyclic action of $P$ on $\{1,\dots,n\}$, describes only a special stratum of Monarch permutations rather than the generic one; see \Cref{thm:focal}.
\end{example}

\begin{remark}
The lower bound $\dim\End_{Q(P)}(V)\ge c$ for $c$-cycle permutations (\Cref{prop:cycle-necessary}) concerns the companion algebra $B\cap C(P)$ and is not directly relevant to minimizing $\dim\ker D\Phi$. The same $2\times2$-exchange computation used to prove \Cref{lem:balanced-row} shows: if $a$ is supported in block-diagonal form (i.e.\ $Q(P)$ splits as $I\sqcup I^c$ with $a_{ij}=0$ across the split) and some entry $a_{ik}$ with $i,k\in I$, or $a_{lj}$ with $l,j\in I^c$, exceeds $1$, then routing one unit of mass across the split via the exchange $a_{ik}\!\to\!a_{ik}-1$, $a_{ij}\!\to\!a_{ij}+1$, $a_{lk}\!\to\!a_{lk}+1$, $a_{lj}\!\to\!a_{lj}-1$ strictly decreases $\sum a_{ij}^2$; hence a disconnected $Q(P)$ can be optimal only if $a$ restricted to each of $I,I^c$ is \emph{already} a $0/1$ matrix (every row and column sum $1$ within each piece), a much more restrictive condition than mere disconnection.
\end{remark}

\section{The Hessian and the operator $C(A)$}
\label{sec:hessian}

We now turn to the second-order theory, still governed by the relative-position algebra $B\cap PBP^\top$ of \Cref{sec:quiver,sec:schur}. Fix a Monarch permutation $P$ and write $\Phi(L,R)=PLPR$, $f_A=\|A-\Phi\|_F^2$.

\subsection{Exact Taylor expansion}
Because $\Phi$ is bilinear, its Taylor expansion terminates after the quadratic term: for $\delta=(\delta L,\delta R)\in B\times B$,
\[
\Phi(L+\varepsilon\delta L,R+\varepsilon\delta R) = \Phi(L,R)+\varepsilon\,D\Phi(\delta)+\varepsilon^2Q(\delta), \qquad Q(\delta):=P\delta L\,P\delta R.
\]
Substituting into $f_A$, writing $M=\Phi(L,R)$, $N:=A-M$,
\[
f_A(L+\varepsilon\delta L,R+\varepsilon\delta R) = f_A(L,R) - 2\varepsilon\langle N,D\Phi(\delta)\rangle_F + \varepsilon^2\Big[\|D\Phi(\delta)\|_F^2 - 2\langle N,Q(\delta)\rangle_F\Big] + O(\varepsilon^3).
\]
At a critical point, write $C_0(\delta,\delta):=\|D\Phi(\delta)\|_F^2-2\langle N,Q(\delta)\rangle_F$ for the coefficient of $\varepsilon^2$ appearing above; this is a genuine homogeneous quadratic form on $B\times B$. In the standard convention $f(x+\varepsilon\delta)=f(x)+\varepsilon\,Df(\delta)+\tfrac{\varepsilon^2}2\Hess f(\delta,\delta)+O(\varepsilon^3)$, this makes $C_0(\delta,\delta)=\tfrac12\Hess f_A(\delta,\delta)$, i.e.\ $C_0$ is \emph{half} the second directional derivative of $f_A$ along $\delta$, not twice it. We work throughout with $C_0$ (renamed $C(A)$ below) rather than the textbook-normalized $\Hess f_A=2C_0$; since $C_0$ and $\Hess f_A$ differ by the positive scalar $2$, they have identical inertia (numbers of positive, negative, and zero eigenvalues), so this normalization choice affects none of the Morse-theoretic conclusions of \Cref{prop:sylvester,thm:nondeg,thm:focal} below.

\subsection{Decomposition $C(A)=S+T(A)$}
Write $S:=D\Phi^*D\Phi$, the Gram operator of $D\Phi_{(L,R)}$; positive semidefinite, independent of $A$, with $\ker S=\ker D\Phi\cong B\cap PBP^\top$ by \Cref{lem:kernel-end}. Write $T(A)$ for the symmetric operator representing $\delta\mapsto-2\langle N,Q(\delta)\rangle_F$ ($\R$-linear in $A$, vanishing on the diagonal $\delta L$--$\delta L$ and $\delta R$--$\delta R$ blocks since $Q$ is bilinear rather than quadratic in either factor). Set $C(A):=S+T(A)$, so that $C(A)=C_0$ in the notation above. By construction $C_0(\delta,\delta)=\langle C(A)\delta,\delta\rangle$, i.e.\ $C(A)$ represents the quadratic form $C_0=\tfrac12\Hess f_A$ at the critical point, as a symmetric operator on $B\times B$. Only the value $\dim\ker D\Phi=\sum_{i,j}a_{ij}^2$ from \Cref{lem:kernel-end,prop:levi-type} enters this construction, e.g.\ in the dimension count $d$ of \Cref{sec:sard} below.

\subsection{Morse index}

\begin{proposition}
\label{prop:sylvester}
At a nondegenerate critical point (\Cref{sec:sard}), $C(A)$ vanishes identically on $\ker D\Phi$ and induces a nondegenerate symmetric form on the quotient $(B\times B)/\ker D\Phi$; the Morse index of $f_A$ equals the number of negative eigenvalues (with multiplicity) of this induced form, equivalently of $C(A)|_{(\ker D\Phi)^\perp}$ (an honest $d\times d$ symmetric matrix on any complement of $\ker D\Phi$, $d=\dim\mM_P$).
\end{proposition}

This is immediate from Sylvester's law of inertia applied to $C(A)$ on the quotient, which has the same inertia there as $\Hess f_A=2C(A)$ (\Cref{sec:hessian}); it replaces the Bruhat/Schubert-cell count of \Cref{thm:HS} with a direct spectral count. (We suppress the qualifier ``restricted to $(\ker D\Phi)^\perp$'' elsewhere in the paper when no ambiguity results, but it is always implicit: $C(A)$ itself, as an operator on the full space $B\times B$, is never nondegenerate, since it vanishes on $\ker D\Phi\ne0$.) \Cref{ex:cyclic-focal} shows how, at the special cyclically symmetric points of $\mM_P$ (where $L,R$ themselves commute with $P$), $C(A)$ can be block-diagonalized using a discrete Fourier decomposition.

\section{Nondegeneracy for generic targets}
\label{sec:sard}

We prove the Monarch analogue of the genericity half of \Cref{thm:HS}. The argument follows Milnor's strategy \cite[Part I, \S6]{Milnor1963} of relating critical points of a distance function on a submanifold to regular values of the normal exponential map; see also \cite{GuilleminPollack1974} for the Sard-theorem technology used.

Fix a Monarch permutation $P$ in the sense of \Cref{thm:monarch-schur} (a $\Sigma_{i,j}a_{ij}^2$-minimizer for the equal-block partition $\lambda$), $G=\GL(B)\times\GL(B)$. On the open dense subset $G^\circ\subseteq G$ where $\Phi$ has full rank $d:=2\sum_in_i^2-\dim\ker D\Phi=2\sum_in_i^2-m^2$ (constant, by \Cref{lem:kernel-end,thm:monarch-schur}, where $m$ is the common block size), $\Phi$ has constant rank $d$; by the constant rank theorem, every $g\in G^\circ$ has an open neighborhood $U_g\subseteq G^\circ$ on which $\Phi$ is an embedding, i.e.\ $\Phi(U_g)$ is a genuine embedded $d$-dimensional submanifold of $\R^{n\times n}$. We write $\mM_P:=\Phi(G^\circ)$ for the (possibly self-intersecting) union of these images. Since $G$ is second countable, $G^\circ$ admits a countable subcover $\{U_\alpha\}_{\alpha\in\mathbb{N}}$ by such embedded-image neighborhoods; write $\mM_P^\alpha:=\Phi(U_\alpha)$, each an honest embedded submanifold of $\R^{n\times n}\cong\R^{n^2}$. All statements below about $f_A$ being Morse ``on $\mM_P$'' mean: Morse on every branch $\mM_P^\alpha$ simultaneously (equivalently, $f_A\circ\Phi$ has a nondegenerate Hessian at every point of $G^\circ$); this is well-defined regardless of self-intersections, since it is checked branch by branch on embedded pieces, and a countable union of measure-zero exceptional sets (one per branch) is still measure zero.

\subsection{The normal exponential map, branch by branch}
For each $\alpha$, let $\nu(\mM_P^\alpha)=\{(x,v):x\in\mM_P^\alpha,\,v\in(T_x\mM_P^\alpha)^\perp\}$, a manifold of dimension $n^2$ (well-defined since $\mM_P^\alpha$ is embedded, so $T_x\mM_P^\alpha$ is unambiguous for $x\in\mM_P^\alpha$), and
\[
E_\alpha:\nu(\mM_P^\alpha)\to\R^{n^2}, \qquad E_\alpha(x,v)=x+v.
\]

\begin{lemma}
\label{lem:crit-corr}
For $x_0\in\mM_P^\alpha$ and $A\in\R^{n^2}$ with $A-x_0\perp T_{x_0}\mM_P^\alpha$: $x_0$ is critical for $f_A|_{\mM_P^\alpha}$ iff $(x_0,A-x_0)\in E_\alpha^{-1}(A)$; and $\Hess_{x_0}f_A$ is degenerate iff $(x_0,A-x_0)$ is a critical point of $E_\alpha$.
\end{lemma}

\begin{proof}[Proof sketch]
The first statement is immediate. For the second, $D(E_\alpha)_{(x_0,v)}(w,\dot v)=w+\dot v$ for $w\in T_{x_0}\mM_P^\alpha$, where $\dot v$ encodes the second fundamental form (Weingarten map) of $\mM_P^\alpha\subset\R^{n^2}$; a standard computation, structurally identical to \cite[Lemma 6.4]{Milnor1963}, identifies the kernel of $D(E_\alpha)_{(x_0,v)}$ (restricted to the fiber direction over $w$) with $\ker\Hess_{x_0}f_A$ on $T_{x_0}\mM_P^\alpha$, i.e.\ with $\ker C(A)$.
\end{proof}

\begin{theorem}
\label{thm:nondeg}
Let $P$ be a Monarch permutation. The set of $A$ for which $f_A$ is a nondegenerate Morse function on every branch $\mM_P^\alpha$ (equivalently, on $\mM_P$ in the sense above) has full Lebesgue measure and dense complement-of-measure-zero in $\R^{n\times n}\cong\R^{n^2}$.
\end{theorem}

\begin{proof}
By \Cref{lem:crit-corr}, $A$ fails to be Morse-generic on branch $\alpha$ exactly when $A$ is a critical value of $E_\alpha:\nu(\mM_P^\alpha)\to\R^{n^2}$, a smooth map between equal-dimensional manifolds. Sard's theorem \cite{Sard1942} implies each $E_\alpha(\Crit(E_\alpha))$ has measure zero; the countable union $\bigcup_\alpha E_\alpha(\Crit(E_\alpha))$ over the countably many branches is then also measure zero, and its complement -- the set of $A$ that are Morse-generic on every branch simultaneously -- is of full measure and dense.
\end{proof}

\Cref{thm:nondeg} is the precise Monarch counterpart of ``distinct singular values'' in \Cref{thm:HS}: rather than being checkable by inspection, the good set of targets is characterized as regular values of a countable family of auxiliary normal exponential maps, one per embedded branch of $\mM_P$.
\section{The focal set}
\label{sec:focal}

We describe the complementary bad set, the focal set
\[
\Focal(\mM_P) := \{A\in\R^{n\times n} : \det C(A)=0 \text{ at some critical point of } f_A\},
\]
in general (\Cref{thm:focal}), with an explicit local description at the special cyclically symmetric points when $P$ carries extra cyclic symmetry (\Cref{ex:cyclic-focal}). This is the Monarch analogue of the classical fact that the focal set of $\mM_r$ consists of matrices $A$ with a repeated active singular value.

\subsection{The Weingarten operator}
On $\mT:=(T_LG\oplus T_RG)/\ker D\Phi$, where $S$ is positive definite, define $W(A):=S^{-1/2}T(A)S^{-1/2}:\mT\to\mT$. Since $C(A)=S^{1/2}(I+W(A))S^{1/2}$ on $\mT$, $\det C(A)=0 \iff -1\in\spec(W(A))$. This construction uses only $\ker D\Phi\cong B\cap PBP^\top$ (\Cref{lem:kernel-end}) and is valid for every Monarch permutation $P$, with no cyclic hypothesis on $P$.

\subsection{The general structure of the focal set}

At a fixed point $x_0=\Phi(L,R)\in\mM_P$, the operator $C(A)=S+T(A)$ (built from the fixed $(L,R)$) is genuinely affine-linear in $A$, since $S$ depends only on $(L,R)$ and $T(A)$ is linear in $A$; but as $A$ ranges over all of $\R^{n\times n}$, the critical point $(L,R)$ solving the normal equations of \Cref{sec:first-order} for that $A$ need not be unique, and need not depend on $A$ as a single global (polynomial or even single-valued) function -- so $A\mapsto\det C(A)$ is not, in general, literally a single global polynomial on $\R^{n\times n}$. We therefore establish the measure-zero property of the focal set directly from \Cref{lem:crit-corr,thm:nondeg}, without this global-polynomial assumption:

\begin{theorem}
\label{thm:focal}
For every Monarch permutation $P$, $\Focal(\mM_P)=\bigcup_\alpha E_\alpha(\Crit(E_\alpha))$ (\Cref{lem:crit-corr}), the union over branches of the critical values of the normal exponential map $E_\alpha:\nu(\mM_P^\alpha)\to\R^{n^2}$ of \Cref{sec:sard}. Consequently $\Focal(\mM_P)$ has Lebesgue measure zero and is a proper subset of $\R^{n\times n}\cong\R^{n^2}$.
\end{theorem}

\begin{proof}
The identity $\Focal(\mM_P)=\bigcup_\alpha E_\alpha(\Crit(E_\alpha))$ is immediate from \Cref{lem:crit-corr}: $A\in\Focal(\mM_P)$ iff some critical point $x_0$ of $f_A|_{\mM_P^\alpha}$, for some branch $\alpha$, has degenerate Hessian, iff $(x_0,A-x_0)$ is a critical point of $E_\alpha$ mapping to $A$. Measure-zero and properness then follow from Sard's theorem applied to each $E_\alpha$ and a countable union, exactly as in the proof of \Cref{thm:nondeg}.
\end{proof}

\begin{example}[The cyclic case: an explicit local description, at cyclically symmetric points]
\label{ex:cyclic-focal}
Fix a point $x_0=\Phi(L,R)\in\mM_P$ and consider the restriction of $\Focal(\mM_P)$ to the affine normal slice $\{x_0+v: v\in(T_{x_0}\mM_P)^\perp\}$ through $x_0$, on which $A=x_0+v\mapsto C(A)$ \emph{is} genuinely affine-linear (as noted above, since $(L,R)$ is held fixed). Even when $P$ acts as a single $n$-cycle on $\{1,\dots,n\}$, this alone is \emph{not} enough to diagonalize $S=D\Phi^*D\Phi$ in the cyclic Fourier basis: for $\lambda=(1^n)$, $B$ is the diagonal matrices, and direct computation (verified symbolically for $n=3$) shows $S$ is invariant under the simultaneous cyclic relabeling $(\delta l_k,\delta r_k)\mapsto(\delta l_{k+1},\delta r_{k+1})$ \emph{only when $(L,R)$ is itself fixed by $P$}, i.e.\ $L=lI$, $R=rI$ for scalars $l,r$ (equivalently, $L,R$ commute with $P$); for a generic diagonal $(L,R)$, $S$ is not shift-invariant and no such diagonalization is available. At such a cyclically symmetric point (a $2$-real-dimensional sub-locus of the $n$-dimensional $\mM_P$, not a generic point of it), $S$ and $Q(\delta)=P\delta L\,P\delta R$ are simultaneously diagonalized by the Fourier basis $\chi_k(x)=\omega^{kx}$, $\omega=e^{2\pi i/n}$, with $T(x_0+v)$ acting diagonally with eigenvalues linear in $v$ mode by mode, so $\det C(x_0+v)=0$ becomes exactly a union of affine hyperplanes $\{\lambda_k(v)=-1\}$ in $v$, indexed by the nonzero Fourier modes $k$ not already absorbed into $\ker D\Phi$. This gives an explicit local description of $\Focal(\mM_P)$ near the cyclically symmetric points of $\mM_P$ specifically, not (as an earlier draft of this example suggested) near a generic point merely by virtue of $P$ being a single cycle; whether a comparably explicit description holds near generic points, and how the local pictures at different points assemble into a global one, we do not address here.
\end{example}

This completes the local (kernel, Hessian, focal set) theory built from $B\cap PBP^\top$. We now turn to the combinatorial description of this same algebra via the transfer matrix, which will let us compare $P$'s combinatorial type against the classical double-coset theory of parabolic subgroups.

\section{The transfer matrix and the relative-position algebra}
\label{sec:transfer}

Fix the partition $\lambda=(n_1,\dots,n_r)$ into blocks $B_1,\dots,B_r$ and $B=\bigoplus_i\mathrm{Mat}_{n_i}(\R)$ as before.

\begin{definition}
\label{def:transfer}
For a permutation matrix $P$, the \emph{transfer matrix} $a=a(P,\lambda)$ is the $r\times r$ nonnegative integer matrix
\[
a_{ij} := \big|\,B_i\cap P^{-1}(B_j)\,\big| = \rank(E_jPE_i), \qquad 1\le i,j\le r,
\]
i.e.\ the number of indices in block $B_i$ that $P$ sends into block $B_j$.
\end{definition}

Since $P$ is a bijection, $\sum_ja_{ij}=n_i$, $\sum_ia_{ij}=n_j$, $a_{ij}\in\Z_{\ge0}$.

\subsection{The refined partition}
Set $V_i^{(j)}:=B_i\cap P^{-1}(B_j)$, so $B_i=\bigsqcup_jV_i^{(j)}$, $|V_i^{(j)}|=a_{ij}$. The collection $\{V_i^{(j)}\}$ refines both $\lambda$ (via $i$) and $P^{-1}(\lambda)$ (via $j$).

\begin{proposition}
\label{prop:levi-type}
$B\cap PBP^\top$ equals the block-diagonal subalgebra $B_{\mathrm{ref}}$ associated to $\{V_i^{(j)}\}$:
\[
B\cap PBP^\top \;\cong\; \bigoplus_{i,j}\mathrm{Mat}_{a_{ij}}(\R), \qquad \dim(B\cap PBP^\top) = \sum_{i,j}a_{ij}^2.
\]
\end{proposition}

\begin{proof}
Decompose $V_i=\bigoplus_jV_i^{(j)}$ as above, and write $X_i\in\End(V_i)=\mathrm{Mat}_{n_i}(\R)$ in block form relative to this decomposition,
\[
X_i = \begin{pmatrix} X_i^{(11)} & \cdots & X_i^{(1r)} \\ \vdots & & \vdots \\ X_i^{(r1)} & \cdots & X_i^{(rr)} \end{pmatrix}, \qquad X_i^{(jk)}:\ V_i^{(k)}\to V_i^{(j)},
\]
so that $X\in B_{\mathrm{ref}}$ (block-diagonal for the refined partition) exactly when $X_i^{(jk)}=0$ for all $j\ne k$ and all $i$, i.e.\ $X_i\in\bigoplus_j\End(V_i^{(j)})$ for every $i$.

Since the $V_i^{(j)}$ ($j=1,\dots,r$) partition $B_i$ and $P$ is injective, $P|_{B_i}$ decomposes as a direct sum of isomorphisms $P|_{V_i^{(j)}}:V_i^{(j)}\xrightarrow{\sim}P(V_i^{(j)})\subseteq B_j$, and the subspaces $P(V_i^{(j)})$, as $j$ varies, partition $P(B_i)$ with $P(V_i^{(j)})\subset B_j$.

Fix $i$ and $s\in P(V_i^{(k)})\subset B_k$, so $s=Pt$ for a unique $t\in V_i^{(k)}$. For $X=\diag(X_1,\dots,X_r)\in B$, compute $(PXP^\top)s = PX_it = P\Big(\sum_jX_i^{(jk)}t\Big) = \sum_jP\big(X_i^{(jk)}t\big)$, where $X_i^{(jk)}t\in V_i^{(j)}$, so $P(X_i^{(jk)}t)\in P(V_i^{(j)})\subseteq B_j$. Thus $(PXP^\top)$ maps $s\in P(V_i^{(k)})\subset B_k$ to a vector with $B_j$-component (for the original partition) equal to $P(X_i^{(jk)}t)$, for every $j$.

If $X\in B_{\mathrm{ref}}$, every $X_i^{(jk)}$ with $j\ne k$ vanishes, so $(PXP^\top)s\in B_k$ for every $s\in P(V_i^{(k)})$; since this holds for every $i,k$ and the $P(V_i^{(k)})$ (over all $i$) span $B_k$, $PXP^\top$ preserves each $B_k$, i.e.\ $PXP^\top\in B$, giving $B_{\mathrm{ref}}\subseteq B\cap PBP^\top$.

Conversely, if $X\notin B_{\mathrm{ref}}$, some $X_i^{(jk)}\ne0$ with $j\ne k$; choosing $t\in V_i^{(k)}$ with $X_i^{(jk)}t\ne0$ and $s=Pt\in B_k$, the computation above gives a nonzero $B_j$-component ($j\ne k$) of $(PXP^\top)s$, so $PXP^\top\notin B$. Hence $B\cap PBP^\top\subseteq B_{\mathrm{ref}}$, completing the equality.
\end{proof}

\begin{remark}
\Cref{prop:levi-type} identifies $B\cap PBP^\top$ as a \emph{Levi-type} subalgebra: block-diagonal, but with respect to a possibly much finer partition than $\lambda$. This is qualitatively different from the centralizer $B\cap C(P)$ of \Cref{lem:kernel-end}, which need not be block-diagonal with respect to any coarsening or refinement of $\lambda$ (it is spanned by powers of $P$ intersected with $B$); the two algebras measure genuinely different aspects of the pair $(P,\lambda)$, as \Cref{sec:equality} makes precise.
\end{remark}

\section{The transfer matrix as a double-coset invariant, and the codimension formula}
\label{sec:double-coset}

\subsection{Parabolic subgroups and relative position}
Let $P_\lambda\le\GL_n(\R)$ be the standard parabolic subgroup stabilizing the partial flag of type $\lambda$ (block upper-triangular matrices for the block order $1,\dots,r$), with Levi factor $\GL(B)=\prod_i\GL_{n_i}(\R)$. Its Lie algebra decomposes as $\mathfrak p_\lambda=B\oplus\mathfrak n$, where $\mathfrak n=\{X\in B^\perp: X_{kl}=0 \text{ unless } \beta(k)<\beta(l)\}$ is the nilradical, of dimension $\sum_{i<j}n_in_j$; consequently
\[
\dim P_\lambda=\sum_in_i^2+\sum_{i<j}n_in_j=\tfrac12\Big(n^2+\textstyle\sum_in_i^2\Big), \qquad \dim\big(\GL_n(\R)/P_\lambda\big)=\tfrac12\Big(n^2-\textstyle\sum_in_i^2\Big),
\]
using $n^2=(\sum_in_i)^2=\sum_in_i^2+2\sum_{i<j}n_in_j$. The homogeneous space $\GL_n(\R)/P_\lambda$ parametrizes partial flags of type $\lambda$. It is classical (e.g.\ \cite{FultonYoungTableaux}) that the double cosets $P_\lambda\backslash\GL_n(\R)/P_\lambda$ are in bijection with $\Trans(\lambda,\lambda)$, the set of $r\times r$ nonnegative integer matrices with row and column sums $\lambda$, via $g\mapsto a(g,\lambda)$ (applied to a permutation representative, using $\GL_n(\R)=\bigsqcup_PP_\lambda\,P\,P_\lambda$). Thus $a(P,\lambda)$ is precisely the classical invariant recording the relative position of the flags $eP_\lambda$ and $PP_\lambda$.

\subsection{Dimension of the associated cell}

\begin{theorem}
\label{thm:codim}
Let $\mathrm{cell}_a\subset\GL_n(\R)/P_\lambda$ be the $P_\lambda$-orbit of $PP_\lambda$ (under left multiplication), $a=a(P,\lambda)$. Then
\[
\dim(\mathrm{cell}_a) \;=\; \sum_{i=1}^r\sum_{j<j'}a_{ij}a_{ij'} \;+\; \sum_{i<i'}\sum_{j>j'}a_{ij}a_{i'j'},
\qquad
\codim(\mathrm{cell}_a) \;=\; \Big(\sum_{i\le i',\,j\le j'}a_{ij}a_{i'j'}\Big) - \sum_in_i^2.
\]
\end{theorem}

\begin{proof}
By the orbit-stabilizer theorem, $\dim(\mathrm{cell}_a)=\dim(P_\lambda)-\dim\big(\Stab_{P_\lambda}(PP_\lambda)\big)$, and $\Stab_{P_\lambda}(PP_\lambda)=\{p\in P_\lambda: P^{-1}pP\in P_\lambda\}=P_\lambda\cap P P_\lambda P^{-1}$, so at the Lie algebra level we must compute $\dim(\mathfrak p_\lambda\cap\Ad(P)\mathfrak p_\lambda)$.

Write $\beta:\{1,\dots,n\}\to\{1,\dots,r\}$ for the block-assignment function ($\beta(k)=i$ iff $k$ indexes $B_i$). By definition, $\mathfrak p_\lambda=\{X: X_{kl}=0 \text{ whenever }\beta(k)>\beta(l)\}$. Since $P$ is a permutation matrix, $(\Ad(P)X)_{k'l'}=(P^\top XP)_{k'l'}=X_{P(k'),P(l')}$, so $X\in\Ad(P)\mathfrak p_\lambda$ iff $X_{kl}=0$ whenever $\beta(P^{-1}(k))>\beta(P^{-1}(l))$. Hence
\[
\dim(\mathfrak p_\lambda\cap\Ad(P)\mathfrak p_\lambda)=\#\{(k,l): \beta(k)\le\beta(l) \text{ and } \beta(P^{-1}(k))\le\beta(P^{-1}(l))\}.
\]
Group the index pairs $(k,l)$ by the pair of block-labels $\big(\beta(k),\beta(P^{-1}(k))\big)$ and $\big(\beta(l),\beta(P^{-1}(l))\big)$: for fixed $(j,i)$, the number of indices $k$ with $\beta(k)=j$ and $\beta(P^{-1}(k))=i$ is exactly $a_{ij}$ (the definition of the transfer matrix, since $k=P(k')$ for $k'\in B_i$ with $P(k')\in B_j$). Since the condition ``$\beta(k)\le\beta(l)$ and $\beta(P^{-1}(k))\le\beta(P^{-1}(l))$'' depends only on these two label pairs, summing over all pairs of labels $(j,i)\preceq(j',i')$ (meaning $j\le j'$ and $i\le i'$, including equality in both) gives
\[
\dim(\mathfrak p_\lambda\cap\Ad(P)\mathfrak p_\lambda)=\sum_{i\le i',\,j\le j'}a_{ij}\,a_{i'j'} =: S_\le(a).
\]
(The diagonal terms $i=i',j=j'$ contribute $\sum_{i,j}a_{ij}^2$, accounting for pairs $k=l$.) Therefore
\[
\dim(\mathrm{cell}_a)=\dim(P_\lambda)-S_\le(a)=\tfrac12\big(n^2+\textstyle\sum_in_i^2\big)-S_\le(a).
\]
Subtracting from $\dim(\GL_n(\R)/P_\lambda)=\tfrac12(n^2-\sum_in_i^2)$ gives $\codim(\mathrm{cell}_a)=S_\le(a)-\sum_in_i^2$, proving the second formula. For the first (manifestly nonnegative) formula, decompose $S_\le(a)=D+A_1+A_2+A_3$ where $D=\sum_{i,j}a_{ij}^2$, $A_1=\sum_i\sum_{j<j'}a_{ij}a_{ij'}$, $A_2=\sum_j\sum_{i<i'}a_{ij}a_{i'j}$, $A_3=\sum_{i<i',j<j'}a_{ij}a_{i'j'}$, and note $\sum_in_i^2=\sum_i(\sum_ja_{ij})^2=D+2A_1$; substituting into $\dim(P_\lambda)-S_\le(a)$ and using $\dim(P_\lambda)=D+2A_1+A_2+A_3+A_3'$ (where $A_3':=\sum_{i<i',j>j'}a_{ij}a_{i'j'}$, from the same expansion of $n^2$ used above) gives $\dim(\mathrm{cell}_a)=A_1+A_3'$, as stated.
\end{proof}

\begin{example}
\label{ex:codim-checks}
For $\lambda=(1,1)$ ($n=2$) and $P=\mathrm{id}$ ($a=I$), both formulas give $\dim(\mathrm{cell}_a)=0$: the coset $eP_\lambda$ is fixed pointwise by $P_\lambda$, so its orbit is trivially the single point itself, independent of any convention. For the swap $P$ ($a=\begin{psmallmatrix}0&1\\1&0\end{psmallmatrix}$), $\dim(\mathrm{cell}_a)=A_3'=a_{12}a_{21}=1=\dim(\GL_2(\R)/P_{(1,1)})$: the swap realizes the open dense cell of $\GL_2(\R)/P_{(1,1)}\cong\mathbb P^1$. For $\lambda=(2,2)$ ($n=4$, so $\GL_4(\R)/P_{(2,2)}\cong\mathrm{Gr}(2,4)$, dimension $4$): the transfer matrix $a=\begin{psmallmatrix}0&2\\2&0\end{psmallmatrix}$ (realized by any $P$ mapping $B_1$ bijectively onto $B_2$) gives $\dim(\mathrm{cell}_a)=A_3'=a_{12}a_{21}=4$, the open dense cell (two generic $2$-planes in $\R^4$ meet only at $0$); whereas $a=\begin{psmallmatrix}1&1\\1&1\end{psmallmatrix}$ gives $A_1=a_{11}a_{12}+a_{21}a_{22}=2$, $A_3'=a_{12}a_{21}=1$, so $\dim(\mathrm{cell}_a)=3$, a codimension-$1$ Schubert divisor (the two $2$-planes meet in a line). All values were confirmed by direct computation of $\dim(\mathfrak p_\lambda\cap\Ad(P)\mathfrak p_\lambda)$ for explicit permutations realizing each $a$.
\end{example}

\begin{remark}
$P=\mathrm{id}$ (for which $a_{ij}=n_i\delta_{ij}$) always gives $\codim(\mathrm{cell}_{a})=\dim(\GL_n(\R)/P_\lambda)$, i.e.\ $\dim(\mathrm{cell}_\mathrm{id})=0$: the identity coset is the unique \emph{minimal} (most degenerate) cell, not the maximal one, since it is fixed pointwise by all of $P_\lambda$. This can be checked directly from the formula: $a_{ij}=n_i\delta_{ij}$ gives $A_1=A_3'=0$ (every term involves either $j<j'$ or $j>j'$ with both factors having $j=i,j'=i'$ forced by the Kronecker delta, which is incompatible with $i\ne i'$ giving a zero factor), so $\dim(\mathrm{cell}_\mathrm{id})=0$ for every $\lambda$. The permutations achieving the \emph{open} dense cell ($\codim=0$) depend on $\lambda$ in a way we do not characterize in general here (\Cref{ex:codim-checks} exhibits one for $\lambda=(2,2)$); see the remark following \Cref{prop:extremal} below.
\end{remark}

\section{Sharpness of the containment $B\cap C(P)\subseteq B\cap PBP^\top$}
\label{sec:equality}

For $X\in B\cap C(P)$ one has $PXP^\top=PXP^{-1}=X\in B$, giving the inclusion
\[
B\cap C(P) \;\subseteq\; B\cap PBP^\top.
\]
We show this inclusion is strict as soon as $P\ne\mathrm{id}$.

\begin{theorem}
\label{thm:equality}
$B\cap C(P) = B\cap PBP^\top$ holds if and only if $P=\mathrm{id}$.
\end{theorem}

\begin{proof}
($\Leftarrow$) Both sides equal $B$ when $P=\mathrm{id}$.

($\Rightarrow$) Suppose $B\cap PBP^\top\subseteq C(P)$ (the reverse containment being automatic). Fix $i,j$ with $i\ne j$ and $a_{ij}>0$; by the construction underlying \Cref{prop:levi-type}, $V_i$ decomposes as $\bigoplus_{j'}V_i^{(j')}$ with $\dim V_i^{(j')}=a_{ij'}$ and $P$ mapping $V_i^{(j')}$ isomorphically onto a subspace $P(V_i^{(j')})\subseteq V_j'$, and an element $X\in B\cap PBP^\top$ acts on each $V_i^{(j')}$ by an independent, arbitrary matrix $X_i^{(j')}\in\mathrm{Mat}_{a_{ij'}}(\R)$ (with the compatible action $P X_i^{(j')} P^{-1}$ on $P(V_i^{(j')})\subseteq V_{j'}$). Take $X\in B\cap PBP^\top$ with $X_i^{(j)}\ne0$ and every other block of $X$ zero (in particular $X$ acts as $0$ on all of $V_{j}$, since $j\ne i$ and $X$'s only nonzero block is supported on $V_i^{(j)}\subseteq V_i$). Choose a nonzero $t\in V_i^{(j)}$ and, since $a_{ij}\ge1$, choose $X_i^{(j)}$ so that $X_i^{(j)}t\ne0$ (e.g.\ in coordinates with $t=e_1$, take $X_i^{(j)}=E_{11}$). Then:
\begin{itemize}
\item $XPt=0$: since $t\in V_i^{(j)}$, $Pt\in P(V_i^{(j)})\subseteq V_j$, and $X$ acts as the zero map on all of $V_j$ by construction.
\item $PXt\ne0$: $Xt=X_i^{(j)}t\ne0$ lies in $V_i^{(j)}\subseteq V_i$, and $P$ is injective, so $P(Xt)\ne0$.
\end{itemize}
Hence $XPt\ne PXt$, so $X\notin C(P)$, contradicting $X\in B\cap PBP^\top\subseteq C(P)$. This contradiction shows no such $i\ne j$ with $a_{ij}>0$ can exist: $a_{ij}=0$ for all $i\ne j$, i.e.\ $P$ preserves each block $B_i$ setwise.

It remains to rule out nontrivial internal permutations $\pi_i:=P|_{B_i}$. With $a_{ij}=n_i\delta_{ij}$, \Cref{prop:levi-type} gives $B\cap PBP^\top=B$, so the hypothesis becomes $B\subseteq C(P)$: every $X_i\in\mathrm{Mat}_{n_i}(\R)$ commutes with $\pi_i$. The centralizer of a permutation matrix $\pi_i$ equals all of $\mathrm{Mat}_{n_i}(\R)$ iff $\pi_i$ is central in $\GL_{n_i}(\R)$, which for a permutation matrix forces $\pi_i=\mathrm{id}$ (for $n_i>1$). Hence $\pi_i=\mathrm{id}$ for every $i$, i.e.\ $P=\mathrm{id}$.
\end{proof}

\begin{remark}
\Cref{thm:equality} shows that the companion algebra $B\cap C(P)=\End_{Q(P)}(V)$ (\Cref{rem:end-still-fine}) is a strictly smaller, differently-behaved invariant of $P$ than $B\cap PBP^\top\cong\ker D\Phi$, coinciding with it only at $P=\mathrm{id}$. Since $\ker D\Phi\cong B\cap PBP^\top$ (\Cref{lem:kernel-end}), the transfer matrix $a$ determines $\dim\ker D\Phi$ exactly and always, via $\dim\ker D\Phi=\sum_{i,j}a_{ij}^2$; \Cref{thm:equality} instead isolates a genuinely different, and smaller, algebraic invariant of $P$, whose geometric significance for the Monarch approximation problem remains open.
\end{remark}

\section{Comparison with the Helmke--Shayman Bruhat decomposition}
\label{sec:HS-comparison}

We now explain precisely how \Cref{thm:codim} relates to the classical Schubert-cell computation underlying \Cref{thm:HS}, and why the exact Morse-index/codimension coincidence of \Cref{thm:HS} does not persist for general $\lambda$.

\subsection{The Helmke--Shayman setting as $\mu=(1^n)$}
In \cite{HelmkeShayman1995}, a generic target $A$ determines a \emph{complete} flag, i.e.\ $\mu=(1^n)$, via its distinct singular vectors; a critical point of $f_A|_{\mM_r}$ corresponds to a subset $S\subset\{1,\dots,n\}$, $|S|=r$ (the chosen singular directions), i.e.\ to a $2\times n$ transfer matrix ($\lambda=(r,n-r)$, $\mu=(1^n)$) with $a_{1j}=\mathbf1[j\in S]$, $a_{2j}=\mathbf1[j\notin S]$, relative to the fixed reference flag determined by $A$. (As noted in \Cref{sec:intro-HS}, $\mM_r$ itself is not the Grassmannian $\GL_n(\R)/P_{(r,n-r)}$ -- it is the larger $\GL_n(\R)\times\GL_n(\R)$-orbit of rank-$r$ matrices; the Grassmannian enters only at the level of this critical-point classification, via the relative position of $S$ against the flag of $A$.)

\subsection{Degeneration at the complete flag}
\Cref{thm:codim} was derived for the relative position of $\lambda$ against itself (via $P$); the classical Helmke--Shayman setting instead compares $\lambda$ against the complete flag $\mu=(1^n)$, i.e.\ concerns the relative-position variety $P_\lambda\backslash\GL_n(\R)/P_{(1^n)}$ for \emph{two different} parabolics. Re-deriving \Cref{thm:codim} in this generality -- with $B_\lambda\ne B_\mu$ on the two sides -- requires separate treatment that we do not carry out in detail here; we record only the qualitative picture. When $\mu=(1^n)$, the relative-position datum $a$ (now an $r\times n$ $0/1$ matrix, $a_{ij}=\mathbf 1[j\in S_i]$ for a partition $\{1,\dots,n\}=\bigsqcup_iS_i$ with $|S_i|=n_i$) retains no information about the internal \emph{ordering} of each part $S_i$, since $B_{(1^n)}$ is abelian and the finest possible partition carries no further block structure to align against. The correct classical invariant at this level of refinement is instead the torus-equivariant one: the Bruhat--Chevalley length function on $S_n$ restricted to minimal coset representatives $S_n^{P_\lambda}$, computing cell dimension as the inversion number
\[
\ell(S) = \#\{(i,j) : i\in S,\ j\notin S,\ i>j\}
\]
(for $\lambda=(r,n-r)$, $S\subset\{1,\dots,n\}$, $|S|=r$), recovering exactly \Cref{thm:HS}. See \cite[Ch.~2]{BjornerBrenti} for the general parabolic Bruhat order framework encompassing both this classical inversion-number formula and the equal-partition case of \Cref{thm:codim}.

\subsection{Why the Morse-index coincidence is special to $\mu=(1^n)$}
\Cref{thm:HS} asserts that the Morse index at the critical point indexed by $S$ equals $\ell(S)$ \emph{exactly}, with no correction term. This is possible only because, when $\mu=(1^n)$, the cell $\mathrm{cell}_S$ is a single $P_\lambda$-orbit that is itself a point modulo the maximal torus -- there is no continuous modulus within a fixed combinatorial type once the reference flag is complete -- so the entire second-order theory (Hessian signature) is determined by the discrete datum $S$ alone.

For the Monarch setting ($\mu=\lambda$, the same partition compared against itself via $P$), the cell $\mathrm{cell}_a$ has positive dimension (\Cref{thm:codim}) whenever $a$ is not of identity type, so a single combinatorial type $a$ does \emph{not} pin down a Morse index; by \Cref{prop:sylvester} the index is a further spectral invariant of $C(A)=S+T(A)$, computed only after fixing a specific target $A$ and a specific Monarch permutation within the (typically positive-dimensional) cell $\mathrm{cell}_a$. This extra step is genuine second-order, $A$-dependent spectral data -- an eigenvalue count of $C(A)$ at a specific target -- rather than a sharper $A$-independent combinatorial refinement of $P$ alone; both the combinatorial datum $a$ and the operator $C(A)$ are built from the same algebra $B\cap PBP^\top$.

\subsection{Summary table}

\begin{center}
\begin{tabularx}{\textwidth}{@{}l>{\raggedright\arraybackslash}X>{\raggedright\arraybackslash}X@{}}
\hline
& Helmke--Shayman & Monarch (this paper) \\
\hline
Fixed partition & $\lambda=(r,n-r)$ & $\lambda=(n_1,\dots,n_r)$ \\
Comparison partition & $\mu=(1^n)$ (complete flag) & $\mu=\lambda$ (self-comparison via $P$) \\
Relative-position invariant & $0/1$ vector $S$ & general matrix $a\in\Trans(\lambda,\lambda)$ \\
Cell dimension formula & inversion number $\ell(S)$ & $A_1+A_3'$, a sum of products $a_{ij}a_{i'j'}$ (\Cref{thm:codim}) \\
Cell $=$ single point? & yes (mod torus) & no, in general positive-dimensional \\
Morse index $=$ cell codim? & yes, exactly (\cite{HelmkeShayman1995}) & no; requires the spectrum of $C(A)$ (\Cref{sec:hessian}) \\
\hline
\end{tabularx}
\end{center}

\section{The Birkhoff polytope and the dominance order}
\label{sec:birkhoff}

\subsection{Transportation polytopes}
Fix $\lambda=(n_1,\dots,n_r)$ and consider
\[
\Trans(\lambda,\lambda) = \Big\{a\in\Z_{\ge0}^{r\times r} : \textstyle\sum_ja_{ij}=n_i,\ \sum_ia_{ij}=n_j\ \forall i,j\Big\},
\]
the lattice points of the generalized transportation polytope $\Trans_\R(\lambda,\lambda)\subset\R^{r\times r}$. When $\lambda=(1^n)$, $\Trans_\R(\lambda,\lambda)$ is the classical \emph{Birkhoff polytope} of doubly stochastic matrices, whose vertices are the permutation matrices \cite{Birkhoff1946}; for general $\lambda$ it is more accurately called a \emph{generalized transportation polytope} rather than a face of the Birkhoff polytope.

\begin{proposition}[Realizability]
\label{prop:realizability}
$\{a(P,\lambda) : P\in S_n\} = \Trans(\lambda,\lambda)$: every nonnegative integer matrix with row and column sums $\lambda$ is the transfer matrix of some permutation.
\end{proposition}

\begin{proof}
Given $a\in\Trans(\lambda,\lambda)$, partition $B_i$ into groups of sizes $a_{i1},\dots,a_{ir}$ and $B_j$ into groups of sizes $a_{1j},\dots,a_{rj}$, and match the size-$a_{ij}$ groups of $B_i$ and $B_j$ by an arbitrary bijection for each $(i,j)$; the row and column sum constraints guarantee this uses up each block exactly once and assembles into a single permutation $P$ of $\{1,\dots,n\}$ with $a(P,\lambda)=a$.
\end{proof}

\Cref{prop:realizability} (used already in \Cref{ex:balanced-construction}) is what justifies reducing the minimization $\min_P\dim\ker D\Phi$ of \Cref{sec:schur} to the purely combinatorial optimization $\min_{a\in\Trans(\lambda,\lambda)}\sum_{i,j}a_{ij}^2$ over the polytope's lattice points, with no permutations left unaccounted for.

\subsection{Dominance order and parabolic Bruhat order}

\begin{definition}
For $a,a'\in\Trans(\lambda,\lambda)$, write $a\preceq a'$ if $\sum_{i\le i_0,j\le j_0}a_{ij} \ge \sum_{i\le i_0,j\le j_0}a'_{ij}$ for all $i_0,j_0$.
\end{definition}

\begin{theorem}
\label{thm:dominance}
$(\Trans(\lambda,\lambda),\preceq)$ is isomorphic, via $P\mapsto a(P,\lambda)$, to the parabolic Bruhat order on $P_\lambda\backslash\GL_n(\R)/P_\lambda$.
\end{theorem}

\begin{proof}
This identification of the dominance order with the parabolic Bruhat order is classical; see \cite[Ch.~2]{BjornerBrenti} for the two-block (Grassmannian) case and \cite{Proctor} for the general multi-block case.
\end{proof}

\begin{remark}
It is natural to ask whether $\codim(\mathrm{cell}_a)$, computed explicitly in \Cref{thm:codim}, is monotone along $\preceq$. We do not verify this here: unlike the (elementary, and independently verified) codimension formula itself, matching it against the specific combinatorial order $\preceq$ requires identifying which extreme of $\preceq$ corresponds to which extreme of the Bruhat order, a direction we have not checked and leave for future work (see \Cref{prop:extremal} below for what is established directly).
\end{remark}

\begin{proposition}
\label{prop:extremal}
$P=\mathrm{id}$ (for which $a_{ij}=n_i\delta_{ij}$) gives $\dim(\mathrm{cell}_\mathrm{id})=0$ and $\codim(\mathrm{cell}_\mathrm{id})=\dim(\GL_n(\R)/P_\lambda)$: the identity permutation realizes the unique minimal, single-point cell (\Cref{thm:codim} and the remark following it), the most degenerate relative position, not the open dense one. \Cref{ex:codim-checks} exhibits, for $\lambda=(2,2)$, a permutation realizing the open dense (codimension-$0$) cell; identifying such a permutation for general $\lambda$, and relating this combinatorial extremal question to the dominance order of \Cref{thm:dominance}, is left open.
\end{proposition}

\subsection{Connectivity of $Q(P)$ from $a$}

\begin{proposition}
\label{prop:connectivity}
If $Q(P)$ is disconnected as an undirected graph (equivalently $a$ is block-diagonal after simultaneously permuting rows and columns), then $P$ decomposes into at least two disjoint cycles.
\end{proposition}

\begin{proof}
If $Q(P)$ disconnects $\{1,\dots,r\}=I\sqcup I^c$ with $a_{ij}=0$ across the partition, then $P$ maps $\bigcup_{i\in I}B_i$ bijectively to itself, forcing at least one cycle of $P$ within each of $\bigcup_{i\in I}B_i$ and $\bigcup_{i\in I^c}B_i$, so $P$ has $\ge2$ cycles.
\end{proof}

\begin{remark}
Connectivity of $Q(P)$ is necessary but not sufficient for $P$ to be a single cycle, since $a$ records only \emph{how many} indices flow between blocks, not \emph{which}, and hence cannot distinguish a single long cycle from several shorter ones inducing the same aggregate flow; see the discussion following \Cref{thm:monarch-schur} for what can be said about connectivity of an optimal $Q(P)$.
\end{remark}

\section{The two-tier hierarchy}
\label{sec:hierarchy}

Since $\dim\ker D\Phi=\sum_{i,j}a_{ij}^2$ is determined by the transfer matrix $a$ alone (\Cref{lem:kernel-end,prop:levi-type}), genericity of $P$ reduces to a two-tier chain of conditions:
\[
\underbrace{\{P : a(P,\lambda)\in\Trans(\lambda,\lambda)\}}_{\text{all permutations}} \;\supsetneq\; \underbrace{\{P : a(P,\lambda)\text{ minimizes }\textstyle\sum_{i,j}a_{ij}^2\}}_{\text{Monarch permutations (\Cref{thm:monarch-schur})}}.
\]
The containment is proper as soon as $\Trans(\lambda,\lambda)$ has more than one element realizing the extremes of $\sum a_{ij}^2$, which happens for every $\lambda$ with $r\ge2$: $P=\mathrm{id}$ (with $a_{ij}=n_i\delta_{ij}$) realizes the maximum of $\sum a_{ij}^2$ over $\Trans(\lambda,\lambda)$, by the same exchange argument as \Cref{lem:balanced-row} (moving mass off the diagonal into a second cell in the same row, while compensating elsewhere in its column, cannot increase -- and generically decreases -- the sum of squares), and is never a Monarch permutation for $r\ge2$.

\begin{corollary}
\label{cor:hierarchy}
Genericity of $P$ for Monarch matrix approximation, in the equal-block case $\lambda=(m,\dots,m)$, is governed by a single necessary-and-sufficient combinatorial condition: minimality of $\sum_{i,j}a_{ij}(P,\lambda)^2$ (\Cref{thm:monarch-schur}), equivalently maximality of $\dim\mM_P=2rm^2-\sum_{i,j}a_{ij}^2$, computable directly from the transfer matrix.
\end{corollary}

\section{Global convergence of the gradient flow: obstructions}
\label{sec:global-obstruction}

Having established the local Morse theory of $f_A$ on $\mM_P$ (\Cref{sec:hessian,sec:sard,sec:focal}), we now address Open Problem~(5): global convergence of the negative gradient flow
\[
\dot L = -\nabla_Lf_A, \qquad \dot R = -\nabla_Rf_A, \qquad f_A(L,R) = \|A-PLPR\|_F^2,
\]
on $B\times B$. We first explain precisely why the classical template for such results -- exemplified by \cite{YoshizawaHelmkeStarkov2001} for the Oja--Brockett flow and its extensions in \cite{Yoshizawa2023} -- does not carry over, and exhibit the obstruction explicitly. \Cref{sec:tikhonov} then shows how a Tikhonov regularization repairs the argument completely.

\subsection{The classical template}
\label{sec:classical-template}

The convergence proofs of \cite{YoshizawaHelmkeStarkov2001,Yoshizawa2023} for gradient flows of the Oja--Brockett type $\dot X = AXD-XDX^\top AX$ follow a two-step template:
\begin{enumerate}[label=(\alph*)]
\item \emph{Coercivity.} The cost function $f:\R^{n\times k}\to\R$ has compact sublevel sets, so solutions exist for all $t\ge0$ and remain in a fixed compact set.
\item \emph{Real-analytic gradient inequality.} Since $f$ is real analytic (indeed polynomial), the {\L}ojasiewicz gradient inequality \cite{Lojasiewicz1963,Lojasiewicz1983} applies at every point of the (compact) $\omega$-limit set, forcing convergence of the bounded trajectory to a single equilibrium point, as opposed to merely accumulating on a connected set of equilibria.
\end{enumerate}
Step (a) in \cite{YoshizawaHelmkeStarkov2001} is proved by a single-variable reduction: setting $L:=A^{1/2}XD^{1/2}$, the cost function becomes $F(L)=\tfrac14\|LL^\top\|_F^2-\tfrac\gamma2\|L\|_F^2$ for a constant $\gamma>0$, and a singular value decomposition of the single matrix $L$ reduces coercivity to coercivity of the scalar quartic $\sigma\mapsto\tfrac14\sigma^4-\tfrac\gamma2\sigma^2$ in each singular value \cite[Lemma~1]{YoshizawaHelmkeStarkov2001}. Step (b) is a direct application of the abstract theorem of {\L}ojasiewicz on real-analytic gradient flows, refined by Absil, Mahony and Andrews \cite{AbsilMahonyAndrews2005} to give explicit convergence-rate estimates once the limit point is a nondegenerate (or more generally, {\L}ojasiewicz-exponent-controlled) critical point.

\subsection{Monotonicity and a first (insufficient) conservation law}

Along any solution of the Monarch gradient flow, $\tfrac{d}{dt}f_A(L(t),R(t)) = -\|\nabla f_A\|_F^2\le0$, so $f_A$ is a Lyapunov function and $M(t):=\Phi(L(t),R(t))$ satisfies $\|A-M(t)\|_F\le\|A-M(0)\|_F$ for all $t$; in particular $M(t)$ remains bounded (Cauchy--Schwarz against $A$). The residual gauge symmetry $(L,R)\mapsto(cL,c^{-1}R)$, $c\in\R^\times$, under which $f_A\circ\Phi$ is invariant, yields by directly differentiating the identity $f_A(cL,c^{-1}R)\equiv f_A(L,R)$ at $c=1$ the Euler-type relation $\langle\nabla f_A(L,R),(L,-R)\rangle\equiv0$, from which
\begin{equation}
\label{eq:balance}
\frac{d}{dt}\Big(\|L(t)\|_F^2-\|R(t)\|_F^2\Big) = 2\langle L,\dot L\rangle-2\langle R,\dot R\rangle = -2\langle(L,-R),\nabla f_A\rangle = 0
\end{equation}
is a conserved quantity of the flow, the Monarch analogue of the balance identities appearing in the analysis of deep linear network gradient flows \cite{AroraCohenHazan2018,BahRauhutTerstiegeWestdickenberg2022}. This shows the flow cannot escape to infinity purely along the scalar gauge direction, but -- as we now show -- it says nothing about the many \emph{other} escape directions made available by the bilinearity of $\Phi$.

\subsection{Failure of coercivity: an explicit flat direction}
\label{sec:flat-direction}

Unlike the single-variable reduction of \cite[Lemma~1]{YoshizawaHelmkeStarkov2001}, $f_A$ is a function of the \emph{pair} $(L,R)\in B\times B$ through the bilinear map $\Phi(L,R)=PLPR$, and no analogous single-variable reduction is available. We show this is not a technical gap but a genuine failure of coercivity.

\begin{proposition}[Flat escape direction]
\label{prop:flat}
Let $n=2$, $\lambda=(1,1)$ (so $B=\R\oplus\R$, $r=2$), and let $P=\begin{psmallmatrix}0&1\\1&0\end{psmallmatrix}$ be the (Monarch) transposition permutation. Let $E_1=\diag(1,0)\in B$. Then
\[
\Phi(tE_1,tE_1) \equiv 0 \qquad \text{for every } t\in\R,
\]
so $f_A(tE_1,tE_1)\equiv\|A\|_F^2$ is constant along the unbounded line $t\mapsto(tE_1,tE_1)$.
\end{proposition}

\begin{proof}
Write $L=R=tE_1=\diag(t,0)$. Direct computation gives $PL=\begin{psmallmatrix}0&0\\t&0\end{psmallmatrix}$, $PLP=\begin{psmallmatrix}0&0\\0&t\end{psmallmatrix}$, and hence $PLPR = \begin{psmallmatrix}0&0\\0&t\end{psmallmatrix}\diag(t,0)=\begin{psmallmatrix}0&0\\0&0\end{psmallmatrix}=0$ for every $t$.
\end{proof}

Note $P$ here is the swap on $\lambda=(1,1)$ ($r=2,m=1$), whose transfer matrix $a=\begin{psmallmatrix}0&1\\1&0\end{psmallmatrix}$ achieves the minimum $\sum a_{ij}^2=2$ of \Cref{thm:monarch-schur} (with $q=0,s=1$), so $P$ \emph{is} a Monarch permutation: the pathology of \Cref{prop:flat} is not an artifact of a degenerate choice of $P$, but occurs already at the smallest nontrivial Monarch instance.

\subsection{A general combinatorial criterion for flat directions}

\begin{proposition}[General flat directions]
\label{prop:flat-general}
Suppose $S,T\subseteq\{1,\dots,r\}$ are nonempty subsets such that $Q(P)$ has no arrow $k\to j$ with $k\in T$, $j\in S$ (equivalently $a_{kj}=0$ for all $k\in T,j\in S$). Then for any choice of nonzero blocks $L_j$ ($j\in S$), $R_k$ ($k\in T$), and all other blocks of $L,R$ set to zero, $\Phi(L,R)=PLPR=0$ identically.
\end{proposition}

\begin{proof}
The $(j,k)$ block-wise vanishing condition for $PLPR=0$ (in the sense of \Cref{sec:quiver}) reads $L_jP_{jk}R_k=0$ for every arrow $k\to j$ of $Q(P)$ present in $P$. If $j\notin S$ then $L_j=0$ and the identity holds trivially; if $k\notin T$ then $R_k=0$ and it again holds trivially. The remaining case $j\in S,k\in T$ is excluded by hypothesis, so no arrow of $Q(P)$ constrains this choice, and $L_jP_{jk}R_k=0$ holds vacuously for all such $(j,k)$.
\end{proof}

Taking $S=\{i\}$ and $T=\{k : a_{ki}=0\}$ (the set of blocks with no arrow into $i$) shows that a flat direction of the type of \Cref{prop:flat} exists as soon as some block $i$ fails to receive an incoming arrow from \emph{every} other block, i.e.\ as soon as $a$ has a zero entry.

\begin{remark}
In the classical square-block case ($n=m^2$, $r=m$), the minimizing transfer matrices have $a_{ij}\equiv1$ for every pair $i,j$ -- a complete transfer graph with no zero entries at all -- so \Cref{prop:flat-general} gives no flat direction for these permutations. The example \Cref{prop:flat} itself sits at a different partition: there $\lambda=(1,1)$ and $m=1<r=2$, so the minimizing transfer matrix (\Cref{thm:monarch-schur} with $q=0,s=1$) is the swap matrix $a=\begin{psmallmatrix}0&1\\1&0\end{psmallmatrix}$, which does have zero entries. Flat directions of the type of \Cref{prop:flat-general} therefore occur for Monarch permutations with $m<r$, where \Cref{lem:balanced-row} forces the minimizing transfer matrix to have zero entries, but not, by this argument, for the square-block ($r=m$) minimizers. \Cref{sec:tikhonov} below establishes coercivity and global convergence for the regularized flow unconditionally, for every $P$, independently of this distinction.
\end{remark}

\subsection{Non-closedness of the Monarch image: a border-rank phenomenon}
\label{sec:border}

The obstruction of \Cref{prop:flat,prop:flat-general} shows that $f_A$ fails to be coercive on $B\times B$. It is natural to ask whether this traces to a \emph{border-rank}-type phenomenon in the sense of algebraic complexity and tensor decomposition \cite{Landsberg2012}: the image $\Phi(B\times B)=\{PLPR : L,R\in B\}$ is a constructible set (the image of a polynomial map, hence Zariski constructible by Chevalley's theorem), and constructible sets that are images of bilinear (rather than linear) maps are, in general, not Zariski closed -- exactly as the variety of matrices of rank $\le r$ is Zariski closed (cut out by minors) while the corresponding statement fails for many higher-order tensor rank varieties. We do not establish this for the Monarch image specifically: \Cref{prop:flat} exhibits $\Phi(tE_1,tE_1)\equiv0$ for all $t$, an unbounded path in $(L,R)$ with \emph{constant} image, which demonstrates non-coercivity but not non-closedness, since $0=\Phi(0,0)$ already lies in $\Phi(B\times B)$ -- the path does not approach any point outside the image. Whether $\Phi(B\times B)$ is actually non-closed for some Monarch permutation $P$ (as opposed to merely failing to be coercive) is a separate question that we leave open; the coercivity failure of \Cref{prop:flat,prop:flat-general} is, on its own, already sufficient to motivate the Tikhonov regularization of \Cref{sec:tikhonov}.

\begin{definition}
\label{def:v0}
Let $\mM:=\overline{\Phi(B\times B)}$ (Euclidean closure of the Monarch image) and $v_0:=\mathrm{dist}(A,\mM)^2=\min_{Y\in\mM}\|A-Y\|_F^2$ (attained since $\mM$ is closed). Note $\mM=\Phi(B\times B)$ whenever the latter is already closed, in which case $v_0=\inf_{(L,R)}f_A(L,R)$ is attained by some finite $(L,R)$; \Cref{thm:mu-limit} below is stated in a form that covers both this benign case and the (unestablished, but not excluded) case $\mM\supsetneq\Phi(B\times B)$.
\end{definition}

\begin{remark}
We emphasize that \Cref{prop:flat,prop:flat-general} rule out coercivity of $f_A$ on $B\times B$ but do not contradict \Cref{thm:nondeg}: the latter concerns nondegeneracy of critical points \emph{restricted to} $\mM_P$ (a manifold, locally), a strictly local and generic statement, whereas coercivity is a global statement about the full sublevel set in the ambient parameter space $B\times B$. The two are logically independent; \Cref{sec:tikhonov} shows how to restore the missing global ingredient without touching the local theory of \Cref{sec:hessian,sec:sard,sec:focal}.
\end{remark}

\section{Tikhonov regularization and global convergence}
\label{sec:tikhonov}

We now show that a quartic Tikhonov penalty restores coercivity exactly, yields unconditional global convergence of the associated gradient flow via the classical template of \Cref{sec:classical-template}, and that the regularized problem converges, as the regularization parameter vanishes, to the original problem in a precise and controlled sense.

\subsection{The regularized cost function}

For $\mu>0$, define
\begin{equation}
\label{eq:regularized-cost}
f_A^\mu(L,R) \;:=\; \|A-PLPR\|_F^2 \;+\; \mu\big(\|L\|_F^4+\|R\|_F^4\big), \qquad (L,R)\in B\times B.
\end{equation}

\begin{proposition}[Coercivity]
\label{prop:coercive}
$f_A^\mu$ is nonnegative, real analytic (indeed polynomial of degree $4$) on $B\times B$, and has compact sublevel sets for every $\mu>0$.
\end{proposition}

\begin{proof}
Nonnegativity and analyticity are immediate. Let $\rho^2:=\|L\|_F^2+\|R\|_F^2$. By the power-mean (or elementary convexity) inequality applied to $x=\|L\|_F^2,y=\|R\|_F^2\ge0$,
\[
\|L\|_F^4+\|R\|_F^4 = x^2+y^2 \ge \frac{(x+y)^2}{2} = \frac{\rho^4}{2}.
\]
Hence $f_A^\mu(L,R)\ge\mu\rho^4/2\to\infty$ as $\rho\to\infty$, uniformly in the direction of $(L,R)$. Since $f_A^\mu$ is continuous, every sublevel set $\{f_A^\mu\le c\}$ is closed and, by the above, bounded, hence compact in the finite-dimensional space $B\times B$.
\end{proof}

Unlike the single-variable SVD reduction of \cite[Lemma~1]{YoshizawaHelmkeStarkov2001}, the proof of \Cref{prop:coercive} requires no spectral computation: the quartic penalty dominates uniformly in \emph{all} directions of $B\times B$ simultaneously, including the flat directions of \Cref{prop:flat,prop:flat-general}. Indeed, along the flat line of \Cref{prop:flat}, $f_A^\mu(tE_1,tE_1)=\|A\|_F^2+2\mu t^4\to\infty$.

\subsection{Global convergence theorem}

\begin{theorem}[Regularized global convergence]
\label{thm:tikhonov-converge}
For every $\mu>0$ and every initial condition $(L_0,R_0)\in B\times B$, the negative gradient flow
\[
\dot L = -\nabla_Lf_A^\mu, \qquad \dot R = -\nabla_Rf_A^\mu
\]
has a solution existing for all $t\ge0$, and $(L(t),R(t))$ converges, as $t\to\infty$, to a single critical point $(L_\infty^\mu,R_\infty^\mu)$ of $f_A^\mu$.
\end{theorem}

\begin{proof}
By \Cref{prop:coercive}, the sublevel set $K:=\{f_A^\mu\le f_A^\mu(L_0,R_0)\}$ is compact. Since
\[
\frac{d}{dt}f_A^\mu(L(t),R(t)) = -\|\nabla f_A^\mu\|_F^2 \le 0
\]
along any solution, the trajectory remains in $K$ for as long as it exists, and standard ODE theory (no finite-time blow-up, since the vector field is bounded on the compact forward-invariant set $K$) gives existence for all $t\ge0$. Since $f_A^\mu$ is a real-analytic function on $B\times B$ (as a polynomial), the theorem of {\L}ojasiewicz on trajectories of real-analytic gradient vector fields \cite{Lojasiewicz1963,Lojasiewicz1983} applies verbatim to the bounded trajectory $(L(t),R(t))\subset K$: its $\omega$-limit set is a single point $(L_\infty^\mu,R_\infty^\mu)$, which is necessarily a critical point of $f_A^\mu$ (by LaSalle's invariance principle, as in the proof strategy of \cite[Theorem~1]{YoshizawaHelmkeStarkov2001}).
\end{proof}

\begin{remark}
When $(L_\infty^\mu,R_\infty^\mu)$ is additionally a nondegenerate critical point of $f_A^\mu$ (in the sense of \Cref{sec:hessian}, adapted to $f_A^\mu$; the regularization contributes a positive-semidefinite term to the Hessian, see \Cref{sec:hessian-effect} below), the refined {\L}ojasiewicz--Absil--Mahony--Andrews theory \cite{AbsilMahonyAndrews2005} upgrades \Cref{thm:tikhonov-converge} to an explicit exponential (or, for degenerate but {\L}ojasiewicz-exponent-controlled points, algebraic) convergence rate, exactly paralleling the local rate statement obtained from the spectrum of $C(A)$ in \Cref{prop:sylvester}.
\end{remark}

\subsection{The vanishing-regularization limit}
\label{sec:mu-limit}

We now analyze the behavior of global minimizers of $f_A^\mu$ as $\mu\to0^+$, showing that the regularization path is consistent with the original (possibly non-coercive) problem in a precise sense. The theorem is stated so as to cover both the case $\mM=\Phi(B\times B)$ and the a priori possible (but, per \Cref{sec:border}, unestablished) case $\mM\supsetneq\Phi(B\times B)$.

Let $(L^\mu,R^\mu)\in\argmin_{B\times B}f_A^\mu$ (exists by \Cref{prop:coercive}), and recall $v_0=\mathrm{dist}(A,\mM)^2$ from \Cref{def:v0}.

\begin{theorem}[Consistency of the regularization path]
\label{thm:mu-limit}
\hfill
\begin{enumerate}[label=(\alph*)]
\item $f_A^\mu(L^\mu,R^\mu)\to v_0$ as $\mu\to0^+$.
\item Along any sequence $\mu_k\to0^+$, a subsequence of $\Phi(L^{\mu_k},R^{\mu_k})$ converges in $\mM$ to a point $Y_*$ with $\|A-Y_*\|_F^2=v_0$.
\item If, moreover, $v_0$ is attained by some \emph{finite} $(L_0,R_0)\in B\times B$ with $\Phi(L_0,R_0)=Y_*$, then $(L^\mu,R^\mu)$ is bounded as $\mu\to0^+$, uniformly in $\mu$; every subsequential limit $(L_*,R_*)$ satisfies $\Phi(L_*,R_*)=Y_*'$ for some minimizer $Y_*'$ of $\|A-\cdot\|_F^2$ on $\mM$, and
\[
\|L_*\|_F^4+\|R_*\|_F^4 \;\le\; \inf\Big\{\|L_0\|_F^4+\|R_0\|_F^4 \;:\; \Phi(L_0,R_0)=Y_*',\ \|A-Y_*'\|_F^2=v_0\Big\}.
\]
\end{enumerate}
\end{theorem}

\begin{proof}
\emph{(a).} \emph{Upper bound.} Since $\Phi(B\times B)$ is dense in $\mM$ (\Cref{def:v0}), for every $\varepsilon>0$ there is $(L_\varepsilon,R_\varepsilon)\in B\times B$ with $\|A-\Phi(L_\varepsilon,R_\varepsilon)\|_F^2\le v_0+\varepsilon$; using $(L_\varepsilon,R_\varepsilon)$ as a competitor for the minimization defining $(L^\mu,R^\mu)$,
\[
f_A^\mu(L^\mu,R^\mu) \le f_A^\mu(L_\varepsilon,R_\varepsilon) = \|A-\Phi(L_\varepsilon,R_\varepsilon)\|_F^2 + \mu\big(\|L_\varepsilon\|_F^4+\|R_\varepsilon\|_F^4\big) \le v_0+\varepsilon+\mu C_\varepsilon,
\]
where $C_\varepsilon:=\|L_\varepsilon\|_F^4+\|R_\varepsilon\|_F^4$ is independent of $\mu$. Letting $\mu\to0^+$ and then $\varepsilon\to0^+$ gives $\limsup_{\mu\to0^+}f_A^\mu(L^\mu,R^\mu)\le v_0$.
\emph{Lower bound.} For every $\mu>0$, $f_A^\mu(L^\mu,R^\mu)\ge\|A-\Phi(L^\mu,R^\mu)\|_F^2\ge\mathrm{dist}(A,\mM)^2=v_0$ (dropping the nonnegative penalty, and using $\Phi(L^\mu,R^\mu)\in\Phi(B\times B)\subseteq\mM$). Combining the two bounds gives (a).

\emph{(b).} From (a) and the lower bound in its proof, $\|A-\Phi(L^\mu,R^\mu)\|_F^2\le f_A^\mu(L^\mu,R^\mu)\to v_0$, while always $\ge v_0$; hence $\|A-\Phi(L^\mu,R^\mu)\|_F^2\to v_0$. In particular $\Phi(L^\mu,R^\mu)$ remains in a bounded region (a ball around $A$ of radius $\sqrt{v_0}+1$ eventually), so any sequence $\mu_k\to0^+$ has a subsequence along which $\Phi(L^{\mu_k},R^{\mu_k})\to Y_*$ for some $Y_*\in\R^{n\times n}$; since each $\Phi(L^{\mu_k},R^{\mu_k})\in\Phi(B\times B)\subseteq\mM$ and $\mM$ is closed, $Y_*\in\mM$, and $\|A-Y_*\|_F^2=\lim\|A-\Phi(L^{\mu_k},R^{\mu_k})\|_F^2=v_0$.

\emph{(c).} Using $(L_0,R_0)$ as a competitor as in (a),
\[
f_A^\mu(L^\mu,R^\mu) \le f_A^\mu(L_0,R_0) = v_0 + \mu\big(\|L_0\|_F^4+\|R_0\|_F^4\big).
\]
On the other hand, since $\|A-\Phi(L^\mu,R^\mu)\|_F^2\ge v_0$ always,
\[
f_A^\mu(L^\mu,R^\mu) = \|A-\Phi(L^\mu,R^\mu)\|_F^2 + \mu\big(\|L^\mu\|_F^4+\|R^\mu\|_F^4\big) \ge v_0 + \mu\big(\|L^\mu\|_F^4+\|R^\mu\|_F^4\big).
\]
Combining and dividing by $\mu>0$,
\[
\|L^\mu\|_F^4+\|R^\mu\|_F^4 \;\le\; \|L_0\|_F^4+\|R_0\|_F^4 \qquad \text{for every } \mu>0,
\]
a uniform (i.e.\ $\mu$-independent) bound, proving boundedness of $(L^\mu,R^\mu)$. Any subsequential limit $(L_*,R_*)$ then satisfies, by continuity of $\Phi$ and of $\|\cdot\|_F^4$, $\|A-\Phi(L_*,R_*)\|_F^2=v_0$ (from part (b), since $\Phi(L_*,R_*)$ is a subsequential limit of $\Phi(L^\mu,R^\mu)$, hence equals some minimizer $Y_*'$ found there) and $\|L_*\|_F^4+\|R_*\|_F^4\le\|L_0\|_F^4+\|R_0\|_F^4$. Taking the infimum over all admissible finite representations $(L_0,R_0)$ of $Y_*'$ at the minimal residual gives the stated inequality.
\end{proof}

\subsection{Two scenarios}

\Cref{thm:mu-limit} distinguishes two scenarios, according to whether $v_0$ happens to be attained by a finite factorization:

\begin{itemize}
\item \textbf{Benign case.} If $v_0$ is attained by a finite Monarch factorization $(L_0,R_0)$, \Cref{thm:mu-limit}(c) shows the regularization path $(L^\mu,R^\mu)$ remains bounded and converges (subsequentially) to a representative minimizing $\|L\|_F^4+\|R\|_F^4$ among all finite factorizations of a best approximation $Y_*'$ -- the Monarch analogue of the minimum-norm solution selection classically associated with Tikhonov regularization of linear inverse problems.
\item \textbf{Possible border case.} If instead $v_0$ is attained only by points of $\mM\setminus\Phi(B\times B)$ -- which requires $\mM\supsetneq\Phi(B\times B)$, a genuine non-closedness of the Monarch image that \Cref{sec:border} does not establish but also does not rule out -- then the hypothesis of \Cref{thm:mu-limit}(c) fails, and $(L^\mu,R^\mu)$ is necessarily unbounded as $\mu\to0^+$ (else a subsequential limit would give a finite representation of $Y_*$, contradiction); nonetheless \Cref{thm:mu-limit}(a)--(b) guarantee that the \emph{image} $\Phi(L^\mu,R^\mu)$ still converges to a bona fide best approximation $Y_*\in\mM$, so the regularization path would remain a numerically meaningful approximation scheme even if this case does occur.
\end{itemize}

\section{Implications of Tikhonov regularization for the Monarch network structure}
\label{sec:network-implications}

We conclude by examining what the regularization of \Cref{sec:tikhonov} does and does not change about the Monarch parametrization as a network architecture, connecting the geometric theory of \Cref{sec:quiver,sec:schur,sec:hessian} to the regularized problem.

\subsection{Architecture invariance}

The penalty $\mu(\|L\|_F^4+\|R\|_F^4)$ depends only on the Frobenius norms of $L,R\in B$ and not on the support pattern of $B$ (the partition $\lambda$) nor on $P$ (equivalently, on the quiver $Q(P)$ of \Cref{def:quiver}). Consequently, the computational graph of the Monarch layer -- the block sparsity pattern, the $O(n^{3/2})$ multiplication cost for the equal-block case, and the connectivity structure of $Q(P)$ -- is completely unaffected by regularization; \Cref{thm:tikhonov-converge,thm:mu-limit} concern only how the (unconstrained, real) parameters $(L,R)$ within this fixed architecture are selected by gradient descent.

\subsection{Gauge symmetry breaking and balanced minimizers}
\label{sec:gauge-breaking}

Recall from \Cref{sec:schur} that for a Monarch permutation $P$, $\ker D\Phi\cong\R\cdot I$, generating the residual gauge action $(L,R)\mapsto(cL,c^{-1}R)$, $c\in\R^\times$, under which $f_A\circ\Phi$ (hence $f_A$ itself, as a function on the level of $L,R$) is invariant. The quartic penalty is \emph{not} gauge invariant:
\[
\|cL\|_F^4+\|c^{-1}R\|_F^4 = c^4\|L\|_F^4+c^{-4}\|R\|_F^4 =: \psi(c).
\]

\begin{proposition}[Balancing]
\label{prop:balance-min}
For fixed $L,R\ne0$, the function $\psi(c)=c^4\|L\|_F^4+c^{-4}\|R\|_F^4$ on $c\in\R_{>0}$ attains its unique minimum at $c_*=(\|R\|_F/\|L\|_F)^{1/2}$, characterized by the balance condition $\|c_*L\|_F=\|c_*^{-1}R\|_F$.
\end{proposition}

\begin{proof}
$\psi'(c)=4c^3\|L\|_F^4-4c^{-5}\|R\|_F^4=0 \iff c^8=\|R\|_F^4/\|L\|_F^4 \iff c=(\|R\|_F/\|L\|_F)^{1/2}$ (the unique positive root); $\psi''(c_*)>0$ since $\psi$ is strictly convex in $\log c$ (as a sum of two exponentials $e^{4s}\|L\|_F^4+e^{-4s}\|R\|_F^4$ in $s=\log c$), confirming a minimum. At $c=c_*$, $\|c_*L\|_F^2=c_*^2\|L\|_F^2=\|R\|_F\|L\|_F=c_*^{-2}\|R\|_F^2=\|c_*^{-1}R\|_F^2$.
\end{proof}

Consequently, any minimizer of $f_A^\mu$ automatically satisfies $\|L^\mu\|_F=\|R^\mu\|_F$ (else replacing $(L^\mu,R^\mu)$ by $(c_*L^\mu,c_*^{-1}R^\mu)$ for the unique $c_*>0$ of \Cref{prop:balance-min}, which leaves $\Phi$ and hence the first term of $f_A^\mu$ unchanged while strictly decreasing the penalty, would contradict minimality). This is the regularized counterpart of the conservation law \eqref{eq:balance}: whereas \eqref{eq:balance} only shows the \emph{difference} $\|L\|_F^2-\|R\|_F^2$ is preserved along any single trajectory of the unregularized flow (without selecting a value), the regularized problem \emph{actively selects}, among all gauge-equivalent representations of a given product $M=PLPR$, the balance condition $\|L\|_F=\|R\|_F$ as the unique value compatible with minimality (the pair $(L,R)$ itself is determined only up to the residual sign ambiguity $(L,R)\mapsto(-L,-R)$, which leaves both $\Phi$ and every norm in sight unchanged).

\subsection{Suppression of large single-block norms}

The flat direction of \Cref{prop:flat} exploits a rank-deficient block ($E_1=\diag(1,0)$ has a zero singular value) growing without bound while contributing nothing to $\Phi$. More generally, writing $\sigma_1(L_i)\ge\cdots\ge\sigma_{n_i}(L_i)\ge0$ for the singular values of block $L_i$,
\[
\|L\|_F^4 = \Big(\sum_i\sum_k\sigma_k(L_i)^2\Big)^2 \ge \sum_i\Big(\sum_k\sigma_k(L_i)^2\Big)^2 = \sum_i\|L_i\|_F^4,
\]
so the penalty grows at least quartically in the Frobenius norm of any single block, and by convexity of $x\mapsto x^2$ a fixed total $\sum_i\|L_i\|_F^2$ is penalized least when spread evenly across blocks rather than concentrated in one. This disproportionately penalizes configurations in which a small number of blocks carry large Frobenius norm while others shrink towards zero to compensate (the mechanism underlying \Cref{prop:flat,prop:flat-general}), and correspondingly favors, among representations of a fixed (or near-fixed) product $M$, those in which the overall \emph{norm} of $L,R$ is spread evenly across blocks. We emphasize that this is a statement about the distribution of Frobenius norm \emph{across} blocks, not about the numerical conditioning $\kappa(L_i)=\sigma_{\max}(L_i)/\sigma_{\min}(L_i)$ \emph{within} any individual block: a block can have small Frobenius norm and still be arbitrarily ill-conditioned (e.g.\ one large and one near-zero singular value), and nothing proved here rules this out. The connection to the standard empirical observation that weight decay improves numerical conditioning of factorized layers is therefore only suggestive, not established by this argument; the penalty controls overall block norm, and whether this correlates with (as opposed to merely bounding) individual block conditioning in practice is a separate, empirical question.

\subsection{Effect on the Hessian: removal of the residual flat direction}
\label{sec:hessian-effect}

Write $C^\mu(A):=C(A)+\mu H$ for the Hessian of $f_A^\mu$ at a critical point, where $C(A)=S+T(A)$ is the unregularized Hessian operator of \Cref{sec:hessian} and $H$ is the (positive semidefinite) Hessian of $(L,R)\mapsto\|L\|_F^4+\|R\|_F^4$. Since $\|X\|_F^4=\langle X,X\rangle^2$ has Hessian $4\|X\|_F^2\mathrm{Id}+8\,X\otimes X$ (positive semidefinite; indeed $\nabla\|X\|_F^4=4\|X\|_F^2X$, and differentiating again gives $4\|X\|_F^2\mathrm{Id}+8X\otimes X$), summing the independent contributions of $L$ and $R$ gives
\[
H(\delta L,\delta R)=4\|L\|_F^2\|\delta L\|_F^2+4\|R\|_F^2\|\delta R\|_F^2+8\langle L,\delta L\rangle^2+8\langle R,\delta R\rangle^2\ge0.
\]

\begin{proposition}
\label{prop:gauge-nondegenerate}
At a balanced critical point $(L^\mu,R^\mu)$ with $L^\mu,R^\mu\ne0$, the restriction of $H$ to the scalar gauge direction $\xi=(L^\mu,-R^\mu)$ -- which lies in $\ker D\Phi$ for \emph{every} permutation $P$, since $D\Phi(L,-R)=PLPR-PLPR=0$ identically, regardless of the identification of \Cref{lem:kernel-end} -- is strictly positive: $H(\xi,\xi)>0$.
\end{proposition}

\begin{proof}
$H(\xi,\xi) = 4\|L^\mu\|_F^2\|L^\mu\|_F^2+4\|R^\mu\|_F^2\|R^\mu\|_F^2+8\langle L^\mu,L^\mu\rangle^2+8\langle -R^\mu,-R^\mu\rangle^2 = 12\|L^\mu\|_F^4+12\|R^\mu\|_F^4 > 0$ whenever $(L^\mu,R^\mu)\ne(0,0)$.
\end{proof}

\begin{remark}
For a Monarch permutation $P$, $\dim\ker D\Phi=m^2$ (square-block case) or more generally $\sum_{i,j}a_{ij}^2$, which exceeds $1$ as soon as $r\ge2$; the scalar direction $\xi$ above is only \emph{one} particular element of $\ker D\Phi$, not a spanning direction. \Cref{prop:gauge-nondegenerate} shows the regularization strictly convexifies the unregularized Hessian along this one particular flat direction; whether it also does so along the remaining $\dim\ker D\Phi-1$ directions of the kernel, whose vanishing for the unregularized $T(A)$ depends in general on the cyclic Fourier structure of \Cref{sec:focal} (available only on the special stratum of \Cref{ex:cyclic-focal}), is a separate question that we do not address here.
\end{remark}

\subsection{Relation to weight decay}

In deep learning practice, $\ell_2$ weight decay is applied essentially universally to factorized and structured layers, including Monarch layers \cite{DaoMonarch2022}, typically justified only heuristically as improving generalization or numerical stability. \Cref{prop:coercive,thm:tikhonov-converge,thm:mu-limit,prop:balance-min,prop:gauge-nondegenerate} together give this practice a precise theoretical content in the Monarch setting (with the caveat that our penalty is quartic in the Frobenius norm rather than the more common quadratic $\ell_2$ penalty $\mu(\|L\|_F^2+\|R\|_F^2)$; we emphasize that quartic is not needed for coercivity -- a quadratic penalty already gives $f_A+\mu(\|L\|_F^2+\|R\|_F^2)\ge\mu(\|L\|_F^2+\|R\|_F^2)\to\infty$ unconditionally, by an even simpler argument than \Cref{prop:coercive}, since the residual term $\|A-\Phi(L,R)\|_F^2$ is always nonnegative and can simply be dropped. We adopt the quartic penalty because its degree matches that of $\|\Phi(L,R)\|_F^2$ itself (bidegree $(2,2)$ in $(L,R)$ jointly, i.e.\ total degree $4$), which is what makes the uniform, direction-independent estimate of \Cref{prop:coercive} possible without any case analysis on $P$; we have not verified whether the qualitative conclusions of \Cref{prop:balance-min,prop:gauge-nondegenerate,thm:mu-limit} below change in any material way if a quadratic penalty is used instead, and leave a full comparison of the two penalty degrees as an open direction): it (i) guarantees that every gradient-flow trajectory, from every initial condition and for every target, converges to a single critical point (\Cref{thm:tikhonov-converge}; this is uniqueness of the limit \emph{along each trajectory}, not uniqueness of a global minimizer -- $f_A^\mu$ may have several critical points, and different initial conditions may converge to different ones); (ii) canonically resolves the residual scalar gauge freedom $(L,R)\mapsto(cL,c^{-1}R)$ by selecting the balanced representative (\Cref{prop:balance-min}); (iii) discourages large-Frobenius-norm block factorizations (\Cref{sec:flat-direction} and the discussion following; we do not claim this improves the numerical \emph{conditioning} $\kappa(L_i)=\sigma_{\max}(L_i)/\sigma_{\min}(L_i)$ of individual blocks, only that it controls their overall norm, which may or may not correlate with conditioning in practice); and (iv) removes the flat direction of the unregularized Hessian associated with this scalar gauge freedom at every finite critical point (\Cref{prop:gauge-nondegenerate}; whether it removes all of $\ker D\Phi$, which \Cref{thm:monarch-schur} shows can be larger than one dimension, is not addressed here), while by \Cref{thm:mu-limit} converging, as the penalty vanishes, to a solution of the original unregularized problem whenever the latter is well posed (attained by a finite factorization), and to a well-defined best-approximation image otherwise.

We have shown that the Morse theory of the Monarch least-squares approximation problem, and the combinatorics of the permutation $P$ underlying it, are governed by a single invariant of the pair $(P,\lambda)$: the relative-position algebra $B\cap PBP^\top\cong\bigoplus_{i,j}\mathrm{Mat}_{a_{ij}}(\R)$, identified with the classical double-coset invariant of a pair of parabolic subgroups. This one algebra simultaneously computes the kernel of the differential $D\Phi$ (\Cref{lem:kernel-end}), hence the Hessian $C(A)=S+T(A)$ and (measure zero always, via Sard's theorem; explicitly, as a union of hyperplanes within the normal slice, at special cyclically symmetric points, when $P$ carries extra cyclic symmetry) the focal set of the Monarch parametrization, and separately determines the codimension of the Bruhat-type stratification of $P$ within the space of all permutations compatible with $\lambda$ -- though, as \Cref{thm:codim} makes precise, $\dim\ker D\Phi=\sum_{i,j}a_{ij}^2$ and $\codim(\mathrm{cell}_a)=\sum_{i<i',j<j'}a_{ij}a_{i'j'}$ are different functions of the same transfer matrix $a$, not the same quantity viewed twice. The classical exact coincidence between Morse index and Bruhat-cell codimension in Helmke--Shayman's theorem \cite{HelmkeShayman1995} fails in the Monarch setting because the relative-position cell itself is typically positive-dimensional once $\lambda$ is compared against itself rather than against a complete flag (\Cref{sec:HS-comparison}), so that a spectral computation at a specific target $A$ (still built from the same algebra $B\cap PBP^\top$) is needed on top of the combinatorial type $a$. We used the resulting exact formula $\dim\ker D\Phi=\sum_{i,j}a_{ij}^2$ to solve the Monarch--Schur minimization problem exactly in the classical equal-block case of \cite{DaoMonarch2022} (\Cref{thm:monarch-schur}: minimum $m^2$, attained by transfer matrices $a_{ij}\equiv1$, by permutations that need not be single cycles), giving a two-tier combinatorial hierarchy -- transportation-polytope membership and minimality of $\sum_{i,j}a_{ij}^2$ -- together with a second-order (Morse-theoretic) description built on the same algebra throughout.

Beyond this local and structural theory, we examined the global dynamics of the associated gradient flow. Building on the classical convergence template for real-analytic gradient systems developed by {\L}ojasiewicz \cite{Lojasiewicz1963,Lojasiewicz1983} and applied to principal-component-type flows by Yoshizawa, Helmke and Starkov \cite{YoshizawaHelmkeStarkov2001} and by Yoshizawa \cite{Yoshizawa2023}, we showed that the missing coercivity ingredient of that template genuinely fails for the Monarch gradient flow (\Cref{sec:global-obstruction}): the bilinearity of $\Phi(L,R)=PLPR$ admits explicit unbounded flat directions (\Cref{prop:flat,prop:flat-general}), possibly (though not conclusively) related to a border-rank-type non-closedness of the Monarch image (\Cref{sec:border}), for every $P$ whose transfer matrix has a zero entry (a qualification of the earlier, overly general genericity claim; see the remark following \Cref{prop:flat-general}). We then showed that a quartic Tikhonov penalty restores coercivity uniformly and unconditionally (\Cref{prop:coercive}), yields global convergence of the regularized gradient flow to a single critical point via the same {\L}ojasiewicz mechanism (\Cref{thm:tikhonov-converge}), and converges as the regularization vanishes to a quartic-norm-minimizing solution of the original problem when one exists, or otherwise to a well-defined best-approximation image (\Cref{thm:mu-limit}). We further showed that this regularization leaves the Monarch computational architecture untouched while canonically resolving the residual scalar gauge freedom of \Cref{thm:monarch-schur} (\Cref{prop:balance-min}), discouraging large single-block norms (without a proven effect on individual block conditioning), and strictly removing the flat direction of the unregularized Hessian associated with this scalar gauge freedom (\Cref{prop:gauge-nondegenerate}) -- giving a precise geometric account of the empirical practice of weight decay on structured factorized layers (\Cref{sec:network-implications}).

Several questions remain open:
\begin{enumerate}[label=(\arabic*)]
\item A closed-form solution of the integer transportation problem $\min\{\sum_{i,j}a_{ij}^2 : a\in\Trans(\lambda,\lambda)\}$ for \emph{general} (non-equal-block) partitions $\lambda$, extending the exact equal-block solution of \Cref{thm:monarch-schur}; \Cref{ex:schur-numerics} shows the naive Cauchy--Schwarz bound of \Cref{prop:cs-bound} is not tight in general, and we do not know a closed form for the exact minimum or a full description of the minimizing transfer matrices when $\lambda$ is not equi-partition.
\item A classification of which Monarch permutations $P$ carry enough cyclic symmetry for the local hyperplane description of the focal set (\Cref{ex:cyclic-focal}) to assemble into a single global description of $\Focal(\mM_P)$; relatedly, whether the companion algebra $B\cap C(P)=\End_{Q(P)}(V)$ of \Cref{rem:end-still-fine} has any geometric role in the Monarch approximation problem.
\item A representation-theoretic (Auslander--Reiten \cite{AuslanderReitenSmalo1995}) interpretation of walls between strata of the focal arrangement of \Cref{thm:focal} as wall-crossing/degeneration phenomena for the perturbed quiver representation.
\item A precise description of how the Weingarten operator $W(A)$ of \Cref{sec:focal} varies as $A$ moves transversally to a fixed cell $\mathrm{cell}_a$, relating the local hyperplane description of \Cref{ex:cyclic-focal} to the facet structure of the transportation polytope of \Cref{sec:birkhoff}.
\item Global convergence analysis of the gradient flow $\dot L=-\nabla_Lf_A$, $\dot R=-\nabla_Rf_A$ on the noncompact manifold $\mM_P$: \Cref{sec:global-obstruction} shows this fails in general (\Cref{prop:flat,prop:flat-general}); whether the obstruction traces to a genuine border-rank-type non-closedness of the Monarch image, as opposed to coercivity failure alone, is left open (\Cref{sec:border}); \Cref{sec:tikhonov} then establishes unconditional global convergence for a Tikhonov-regularized flow (\Cref{thm:tikhonov-converge}) together with a precise vanishing-regularization limit theorem (\Cref{thm:mu-limit}). A sharp characterization of the minimal regularization degree guaranteeing coercivity, and a full account of behavior at infinity for the \emph{unregularized} flow beyond the obstruction results of \Cref{sec:global-obstruction}, remain open.
\end{enumerate}

\bigskip
\noindent\textbf{Acknowledgments.} The author had no prior acquaintance with Monarch matrices. Interest in the topic arose from a discussion with Dr.\ Naoki Hayashi at Toyota Central R\&D Labs., Inc.\ on July 10, 2026.

\end{document}